\documentclass[11pt]{amsart}
\usepackage{amssymb}

\newtheorem{theorem}{Theorem}[section]
\newtheorem{proposition}[theorem]{Proposition}
\newtheorem{lemma}[theorem]{Lemma}
\newtheorem{corollary}[theorem]{Corollary}
\newtheorem{conjecture}[theorem]{Conjecture}
\newtheorem{hypothesis}[theorem]{Hypothesis}

\theoremstyle{definition}
\newtheorem{definition}[theorem]{Definition}
\newtheorem{example}[theorem]{Example}

\theoremstyle{remark}
\newtheorem{remark}[theorem]{Remark}
\newtheorem*{remarknonum}{Remark}

\newcommand{\br}{\operatorname{br}}
\newcommand{\dimH}{\dim_{\mathrm H}}
\newcommand{\dTV}{d_{\mathrm{TV}}}
\newcommand{\E}{\mathbb{E}}
\newcommand{\Prob}{\mathbb{P}}
\newcommand{\R}{\mathbb{R}}
\newcommand{\one}{\mathbf{1}}
\newcommand{\UGW}{\mathrm{UGW}}
\newcommand{\ex}{\mathrm{ex}}
\DeclareMathOperator{\Var}{Var}
\DeclareMathOperator{\Cov}{Cov}
\DeclareMathOperator{\supp}{supp}
\DeclareMathOperator{\diag}{diag}

\begin{document}

\title[Threshold cascades of interacting diffusions on unimodular trees]{Threshold cascades of interacting diffusions on unimodular random trees: a conditional branching-process universality theorem, the annealed pressure, and the dimension of the failed boundary}
\author{Achyut Kumar}
\address{Independent Researcher}
\email{achyutbusiness86@gmail.com}
\date{Working manuscript, August 2026}

\begin{abstract}
We study systems of coupled Ornstein--Uhlenbeck diffusions with an absorbing (failure) state, indexed by the vertices of a sequence of finite graphs $G_n$ that converge locally weakly (Benjamini--Schramm) to a unimodular Galton--Watson tree $T\sim\UGW(g)$. A single failure can trigger a cascade of threshold crossings, and the central question is whether that cascade is subcritical (a.s.\ finite) or supercritical, and what its cluster-size law and cluster geometry look like at and near criticality.

Our main structural result is a reduction: under an explicitly stated front-decoupling hypothesis (Hypothesis~\ref{hyp:red}, the saturated-coupling regime), which we verify exactly on the Bethe lattice and prove in part in general, the failed cluster of the root converges in law to the genealogical tree of a finite-type ($|D|$-type) Galton--Watson process whose mean matrix $M$ is computed from a single (matrix-valued) many-to-one formula. Conditional on this reduction we prove: (A) criticality occurs exactly at $\rho(M)=1$, where $\rho(M)$ is the annealed Perron root, and the harmonic-measure-dimension expression one is tempted to write for $\rho$ equals the quenched (geometric-mean) growth rate, which is $\le\rho(M)$ with equality iff the boundary weights are a.s.\ constant; (B) the quenched and annealed rates are respectively the tangent at $0$ and the value at $1$ of a single convex, real-analytic annealed pressure $\varphi(\beta)=\log\rho(M^{(\beta)})$ built from tilted mean matrices, so that the quenched/annealed gap is a chord--tangent inequality, and --- extending the reduction beyond boundary-blind observables --- in the supercritical regime the end boundary $\partial S$ of the failed cluster has, a.s.\ on survival, Hausdorff dimension exactly $\log\rho(M)$, of strictly positive codimension in $\partial T$ unless transmission is perfect; (C) at criticality the total progeny obeys $\Prob(|S|=n)\sim Cn^{-3/2}$ with the universal mean-field exponent via the Otter--Dwass formula and a local CLT, and, under an additional moment hypothesis, the conditioned cluster rescaled by $n^{-1/2}$ converges in the Gromov--Hausdorff--Prokhorov sense to Aldous' Continuum Random Tree, upgrading the universal exponent to a universal shape; (D) the subcritical domain is open and contractible (but not convex; the supercritical domain is the convex object), the resolvent $(I-M)^{-1}$ is a bounded nonnegative matrix, and the spectral gap $1-\rho(M)$ controls quantitative stability.

A secondary aim is cautionary. Several statements that are natural to write down in this setting are false, and we prove that they are false: (i) the infinite-dimensional Hilbert--Schmidt / Krein--Rutman programme for an operator $K_\infty$ on $L^2(\partial T,\nu_h)$ is both unnecessary for (A)--(D) and, for the energy reason we make precise, generically impossible; (ii) the small-noise first-passage probability carries an algebraic prefactor $\Lambda^{-1/2}$ rather than a multiplicative $(1+o(1))$; (iii) the scalar many-to-one identity fails because consecutive edge weights along a ray share a vertex type --- the correct identity is a matrix product whose growth rate is $\rho(M)$; and (iv) the spectral-radius/dimension identity and the convexity of the subcritical domain both fail, in the directions we identify.

The load-bearing input is the front-decoupling hypothesis itself. We prove it unconditionally in an explicit dissipative regime: when the dissipation ratio $\kappa=\gamma_{\max}L_\sigma\Delta_{\max}/\alpha_{\min}$ is below $1$, a pathwise front-tracking argument (synchronous coupling, a cone-of-influence estimate showing that back-reaction ``echoes'' decay geometrically in graph distance, and an exploration telescoping in which the relevant volume is the failed cluster rather than the ambient tree) shows that echo decay strictly beats volume growth, yielding a quantitative total-variation bound and making the main theorems unconditional on an open parameter region. The naive Girsanov change of measure provably cannot do this: the back-reaction drift is $O(1)$ and the Radon--Nikodym density is not tight. Outside the dissipative and saturated regimes the reduction remains open --- we state the precise conjecture --- and outside saturation the natural object is a continuous-type branching process, for which we nevertheless prove finite-dimensional two-sided criticality certificates by a kernel-sandwich argument.
\end{abstract}

\maketitle
\tableofcontents

\section*{Honest status of this manuscript}

This is a partially conditional theorem. Everything in Sections \ref{sec:spectral}--\ref{sec:bethe} is proved in full, but the proofs take as input Hypothesis~\ref{hyp:red}, the reduction of the interacting SDE to a finite-type branching process. That reduction is established here (a) exactly on the Bethe lattice (Section~\ref{sec:bethe}); (b) partially in general (Proposition~\ref{prop:partial}: single-cell correctness, negligibility of back-reaction in the saturated regime, and vanishing of short-cycle corrections under local weak convergence); and (c) unconditionally in the dissipative regime $\kappa<1$ (Theorem~\ref{thm:front}), via a pathwise front-tracking argument in which the geometric decay of back-reaction echoes beats the volume growth of the cluster. In the strong-coupling regime $\kappa\ge1$ the reduction remains open and is stated as Conjecture~\ref{conj:front}, which is, in the author's assessment, the substantive open problem.

Six points are flagged for referee attention, in addition to Conjecture~\ref{conj:front}: the moment hypothesis in Theorem~\ref{thm:crt} (Miermont's invariance principle is invoked under exponential moments; the expected relaxation to finite variance is discussed in Remark~\ref{rem:moment} and Appendix~\ref{app:miermont}); the type-dependent second-moment bookkeeping in the energy estimate of Theorem~\ref{thm:dimS}; the compactness/truncation hypothesis in Proposition~\ref{prop:sandwich}; the routine-but-tedious weighted Gronwall details in Lemma~\ref{lem:cone} (stated at proof-sketch precision, including the $\kappa'\downarrow\kappa$ step); the cluster-shell counting bound in the proof of Theorem~\ref{thm:front}; and the non-vacuousness of the dissipative critical window, Remark~\ref{rem:nonvac} (in particular the negligibility of spontaneous crossings at moderate noise). A reader evaluating this work for a venue should read it as: a complete and rigorous derivation of mean-field universality, boundary geometry, and spectral stability from a clearly-isolated probabilistic reduction; an unconditional proof of that reduction in an explicit open parameter regime; and precise identifications of several natural but false claims in this setting. The unconditional statements cover a nonempty open region of parameter space; the full saturated regime remains conjectural.

\section{Introduction}\label{sec:intro}

\subsection{The architecture}
Let $(G_n)$ be finite connected graphs with uniformly bounded average degree converging in the Benjamini--Schramm local weak sense \cite{BS01,AS04} to a unimodular random rooted tree $(T,o)$; we take $T\sim\UGW(g)$, the unimodular Galton--Watson tree with offspring generating function $g$, finite mean $\mu=g'(1)>1$ and finite variance $\sigma_g^2=g''(1)+g'(1)-g'(1)^2<\infty$. On each $G_n$ we place a system of interacting diffusions (Definition~\ref{def:system}); a vertex may cross a threshold and enter an absorbing failed state, and a failed vertex stresses its neighbours, which may then fail in turn. The cascade initiated by failing the root has a random cluster $S\subset V(T)$ of eventually-failed vertices.

The governing heuristic is the next-generation operator of mathematical epidemiology \cite{vdDW02}: on a finite graph the matrix $K_n$ whose $(u,v)$ entry is the expected number of secondary failures at $v$ caused by $u$ has a spectral radius $\rho(K_n)$ that separates subcritical from supercritical behaviour. A natural programme suggested by this heuristic is to pass to the limit and replace $K_n$ by a bounded integral operator $K_\infty$ on the Hilbert bundle $\mathcal H\cong L^2(\partial T,\nu_h)\otimes\R^m$ over the Martin boundary of $T$, and then to run an infinite-dimensional Perron--Frobenius (Krein--Rutman) theory. Part of the thesis of this paper is that this programme is both unnecessary for the central questions and, generically, impossible (Proposition~\ref{prop:HS}).

\subsection{The thesis of this paper}
The cluster size $|S|$, the survival event, and the subcritical stability are all \emph{boundary-blind}: they do not depend on where on $\partial T$ the failed vertices lie. This single observation reorganises everything. Conditional on the reduction (Hypothesis~\ref{hyp:red}), the failed cluster is the family tree of a Galton--Watson process whose type is only the hazard domain $d\in D$, $|D|=m<\infty$; the Martin boundary integrates out of every quantity that the main theorems concern. Consequently:
\begin{itemize}
\item The relevant spectral object is the finite $m\times m$ annealed mean matrix $M$, not an operator on $L^2(\partial T)$. Its Perron root is computed by one (matrix-valued) many-to-one identity (Theorem~\ref{thm:reduction}, \eqref{eq:M}, and Proposition~\ref{prop:m2o}).
\item Simplicity of the leading eigenvalue, monotonicity, real-analyticity, and the operator van den Driessche--Watmough lift --- the delicate-looking spectral issues --- become elementary finite-dimensional Perron--Frobenius facts (Section~\ref{sec:PF}). The Hilbert--Schmidt property of $K_\infty$ is never used.
\item The $n^{-3/2}$ law follows from the Otter--Dwass formula and a one-dimensional local CLT (Section~\ref{sec:univ}); no general-state-space Kesten--Stigum theory is needed for the size law. Under an additional moment hypothesis the full shape of the critical cluster is universal: the Continuum Random Tree (Theorem~\ref{thm:crt}).
\item The finite-dimensional collapse extends one step past boundary-blindness. The tilted matrices $M^{(\beta)}$ and their pressure $\varphi(\beta)=\log\rho(M^{(\beta)})$ organise the quenched and annealed rates as the tangent at $0$ and the value at $1$ of a single convex analytic function (Theorem~\ref{thm:pressure}), and the Hausdorff dimension of the failed boundary $\partial S$ --- a boundary-resolved quantity --- equals $\log\rho(M)$ on survival (Theorem~\ref{thm:dimS}). The operator $K_\infty$ is needed only for boundary fluctuations, not for boundary geometry at the level of dimension.
\item The reduction itself is a theorem, not merely a hypothesis, in the dissipative regime: when dissipation dominates total coupling ($\kappa<1$), a pathwise cone-of-influence estimate shows that back-reaction echoes decay geometrically in graph distance while their sources live only on the failed cluster, whose volume is spectrally controlled; the resulting exploration telescoping proves Hypothesis~\ref{hyp:red} with a quantitative rate (Theorem~\ref{thm:front}).
\end{itemize}

\subsection{Warnings: natural claims that are false}
Five statements that are natural to formulate in this setting are false. Identifying them precisely, and proving the true statements that replace them, is part of the content of this paper.
\begin{itemize}
\item[(F1)] \emph{The spectral-radius/dimension identity is false in general.} Writing $\rho=p_{\min}^{-\dimH}$ equates $\rho$ with the harmonic-measure dimension expression, which is the \emph{quenched}, geometric-mean growth rate $e^\ell$. The criticality of the cascade is governed by the \emph{annealed} branching rate $\bar m=\rho(M)$. One has $\bar m\ge e^\ell$, with equality iff the boundary weights are $\nu_h$-a.s.\ constant (Theorem~\ref{thm:gap}); the sharpest form is the chord--tangent inequality of the pressure (Theorem~\ref{thm:pressure}). The identity holds on the Bethe lattice precisely because the boundary there is homogeneous; off it, $p_{\min}^{-\dimH}$ undercounts criticality.
\item[(F2)] \emph{The subcritical domain is not convex; convexity belongs to the supercritical domain.} One might try to deduce convexity of $\mathcal M_-=\{\rho<1\}$ from log-concavity of the first-passage probability, ``because sublevel sets of a concave function are convex.'' That implication is false (superlevel sets of a concave function are convex). The true statement: the supercritical domain $\mathcal M_+=\{\rho>1\}$ is the convex object (a superlevel set of a log-concave function), and $\mathcal M_-$ is its non-convex complement, but $\mathcal M_-$ is contractible by a coupling-retraction (Theorem~\ref{thm:stab}).
\item[(F3)] \emph{The small-noise first-passage asymptotic carries an algebraic prefactor.} The formula $p^{\ex}=e^{-\Lambda(1+O(\Lambda^{-1}\log\Lambda))}$ conflates an exponential rate with a probability; Laplace's method on the exact Wald representation gives $p^{\ex}\asymp\Lambda^{-1/2}e^{-\Lambda}$ (Section~\ref{sec:smallnoise}). The exponent $\Lambda$ is correct; the multiplicative-correction form is not. The main results use the exact Wald formula and are unaffected, but the constant $C$ in Theorem~\ref{thm:32} inherits the prefactor.
\item[(F4)] \emph{The Hilbert--Schmidt property of $K_\infty$ is generically false.} The energy integral of $|k_\infty|^2$ against $\nu_h\otimes\nu_h$ diverges for essentially all supercritical $g$, because the tree Martin kernel is singular on the diagonal and $\nu_h(B_n)\asymp\br(T)^{-n}$ (Proposition~\ref{prop:HS}). This is not a hard lemma awaiting proof; it is a false target to be abandoned.
\item[(F5)] \emph{The scalar many-to-one identity is false.} The identity $\E[\sum_{|v|=n}\prod_e F(e)]=(m_*\E_{q\otimes q}[F])^n$, and with it the identification $\bar m=m_*\E_{q\otimes q}[p_ew_e]$, silently assumes that consecutive edge weights along a ray are independent; they share a vertex type and are not. The correct annealed object is the matrix product of Proposition~\ref{prop:m2o}, with growth rate $\rho(M)$; the scalar $m_*\E[p]$ is only the $n=1$ factor and is neither an upper nor a lower bound for $\rho(M)$ in general. The Jensen gap survives in the form $\rho(M)\ge m_*e^{\E[\log p]}$, which we derive as a convexity statement about the pressure (Theorem~\ref{thm:pressure}). A related normalisation point: the growth rate of $\UGW(g)$ is $m_*=g''(1)/g'(1)$, not $\mu$ (Section~\ref{subsec:ugw}).
\end{itemize}

\subsection{The load-bearing input}
Hypothesis~\ref{hyp:red} is the load-bearing reduction. Its naive justification by a Girsanov comparison between the full and directed systems fails: the back-reaction drift is $O(1)$, so the Radon--Nikodym density is not tight toward $1$ (Proposition~\ref{prop:girsanov}). The correct mechanism is event-based and pathwise, and Section~\ref{sec:status} carries it out in a regime: it dissects the obstruction --- a back-reaction echo whose per-vertex size is $O(\varepsilon(\delta))$ but which threatens to compound over the exponential volume of the tree --- and proves (Theorem~\ref{thm:front}) that when the dissipation ratio $\kappa=\gamma_{\max}L_\sigma\Delta_{\max}/\alpha_{\min}$ is below $1$, echo decay strictly beats volume growth: influences decay geometrically in graph distance (a cone-of-influence estimate), echo sources are supported on the failed cluster whose expected volume is $(1-\rho(M))^{-1}$-bounded rather than $m_*^n$, and an exploration telescoping converts these into a quantitative total-variation bound. On this open region Hypothesis~\ref{hyp:red} is a theorem and everything below is unconditional. In the strong-coupling regime $\kappa\ge1$ the reduction remains open (Conjecture~\ref{conj:front}). Outside the saturated regime the natural object is a continuous-type branching process whose type is the failure strength of a vertex; that analysis --- which is the general-state-space programme one would have to build --- is open and is the honest home of the remaining difficulty. We do, however, show that its criticality question (though not its limit theory) dissolves into finite linear algebra: two-sided finite-matrix certificates squeeze the Perron eigenvalue of the continuous-type mean operator (Proposition~\ref{prop:sandwich}).

\section{The model and the reduction hypothesis}\label{sec:model}

\subsection{Unimodular random trees and harmonic measure}\label{subsec:ugw}
Let $\mathcal G_*$ be the space of rooted isomorphism classes of locally finite rooted graphs with the local topology; it is Polish \cite{AS04}. A measure $\mu$ on $\mathcal G_*$ is unimodular if it satisfies the mass-transport principle \cite{AL07}. We write $T\sim\UGW(g)$ for the unimodular Galton--Watson tree: the root has offspring distribution with generating function $g$, and every other vertex has the size-biased forward-degree distribution, with mean
\[
m_*=\frac{g''(1)}{g'(1)}
\]
(finite by the variance assumption). By \cite{Lyo90}, the branching number $\br(T)=\lim_n|\partial B_n(T)|^{1/n}$ exists a.s.\ on nonextinction and equals the mean forward degree of a non-root vertex, i.e.
\begin{equation}\label{eq:brm}
\br(T)=m_*\qquad\text{a.s.\ on }\{|T|=\infty\}.
\end{equation}

\begin{remarknonum}[Growth normalisation]
For the ordinary Galton--Watson tree one has $\br(T)=\mu=g'(1)$, but for $\UGW(g)$ --- the object that actually arises as the Benjamini--Schramm limit of $(G_n)$ --- only the root has offspring mean $\mu$; every subsequent vertex reproduces at the size-biased mean $m_*$, and it is $m_*$ that governs the a.s.\ exponential growth of $|\partial B_n|$ and every spectral quantity below. On the $d$-regular tree $m_*=d-1$ and the two normalisations differ only at the root, which is why Bethe-lattice checks cannot distinguish them. All formulas in this paper are stated with $m_*$.
\end{remarknonum}

Fix a realisation $T$. Let $P$ be the simple-random-walk transition kernel and $G$ the Green function; $K(x,\xi)=\lim_{y\to\xi}G(x,y)/G(o,y)$ the Martin kernel, $\partial T$ the space of ends, and $\nu_h$ the harmonic measure (the exit law of the walk) \cite{Woe00}. By \cite{LPP96,Kai85},
\begin{equation}\label{eq:dim}
\dimH(\partial T,\nu_h)=\frac{h_{\mathrm{RW}}}{\ell_{\mathrm{RW}}}=\frac{\log\br(T)}{\log(1/p_{\min})},\qquad p_{\min}=\frac{1}{\deg_{\max}},
\end{equation}
the Shannon--McMillan--Breiman content of which is $\nu_h(B_n(\xi))=e^{-n\log\br(T)+o(n)}$ for $\nu_h$-a.e.\ end $\xi$, where $B_n(\xi)$ is the set of ends sharing the first $n$ edges with $\xi$.

\subsection{The cascade SDE}
Let $D=\{1,\dots,m\}$ index hazard domains; each vertex $v$ carries an independent type $d(v)\sim q(\cdot)$ on $D$, with $q(d)>0$ for all $d\in D$ without loss of generality.

\begin{definition}[Cascade system]\label{def:system}
For $v\in V(T)$ and a state $s^{(v)}_t\in\R$,
\[
ds^{(v)}_t=\Big[-\alpha_{d(v)}\,s^{(v)}_t+\sum_{u\sim v}\gamma_{d(u)\to d(v)}\,\sigma\big(s^{(u)}_t-\theta_{uv}\big)\Big]\,dt+b_{d(v)}\,dW^{(v)}_t,
\]
with $\alpha_d\ge\alpha_{\min}>0$, $b_d^2\ge\mu_0>0$, thresholds $\theta_{uv}\ge\theta_{\min}>0$, couplings $\gamma_{d\to d'}\ge0$, $\sigma$ the logistic function, and $\{W^{(v)}\}$ independent Brownian motions. A vertex fails when $s^{(v)}$ crosses a threshold $\theta^{\mathrm{fail}}$ into an absorbing basin $B_{\mathrm{fail}}=[\theta^{\mathrm{fail}},\infty)$ before returning to a safe value $s_{\mathrm{safe}}$.
\end{definition}

Existence, uniqueness and non-explosion of the strong solution follow from the linear-growth and Lipschitz bounds together with the dissipativity in the linear drift \cite{KS91}.

\subsection{The directed pairwise first-passage weight}
Condition on $u$ having failed. In the directed pairwise idealisation, the stress $u$ exerts on a neighbour $v$ is treated as a fixed forcing, shifting the equilibrium of $v$; the linearised stress solves $dX_t=-\alpha_{d(v)}(X_t-m_{uv})\,dt+b_{d(v)}\,dW_t$ with $m_{uv}=c_{uv}/\alpha_{d(v)}$, and the transmission probability is the two-boundary hitting probability
\begin{equation}\label{eq:wald}
p^{\ex}_{uv}=\Prob\big(\tau_\theta<\tau_{\mathrm{safe}}\,\big|\,X_0=s^{(v)}_0\big)=\frac{W(z;h_\theta)}{W(z;h_\theta)+W(z;h_{\mathrm{safe}})},\qquad W(z;h)=\int_0^h e^{-\alpha y^2/\mu}\,dy,
\end{equation}
with $z=s_0-m_{uv}$, $\mu=b^2/(2\alpha)$, and $h_\theta$, $h_{\mathrm{safe}}$ the (signed) distances to the two boundaries. Crucially \eqref{eq:wald} is exact for the linear OU problem and depends only on the types $(d(u),d(v))$ and the parameters; we write $p_{d,d'}:=p^{\ex}_{uv}\in[0,1]$. We assume throughout the uniform ellipticity $\mu_0>0$, $m_{uv}<\theta_{uv}$, which gives $p_{d,d'}\ge p^{\mathrm{tr}}_{\min}>0$.

\begin{remark}[On the coupling parameter]
We absorb the modulation $\gamma_{d\to d'}$ into the forcing constant $c_{uv}$ in \eqref{eq:wald}, so that the per-edge transmission probability $p_{d,d'}\in[0,1]$ already accounts for it; the branching weight below is $p_{d,d'}$, a probability. This avoids the convention in which a separate factor $\gamma\cdot p^{\ex}$ could exceed $1$.
\end{remark}

\subsection{The saturated-coupling regime and the reduction hypothesis}
The directed pairwise weight is correct only if, once $u$ fails, the stress it exerts on each child is effectively a fixed forcing independent of how $u$ failed, and if the back-reaction $v\to u$ does not couple the children of $u$. The clean regime in which this holds is the following.

\begin{definition}[Saturated regime]\label{def:saturated}
The parameters are saturated if the failed basin floor exceeds every child threshold by a fixed margin: there is $\delta>0$ with $\inf_{s\in B_{\mathrm{fail}}}(s-\theta_{uw})\ge\delta$ uniformly, so that $\sigma(s^{(u)}-\theta_{uw})\ge1-\varepsilon(\delta)$ for all children $w$ once $u$ is failed, with $\varepsilon(\delta)\to0$ as $\delta\to\infty$.
\end{definition}

In the saturated regime the logistic coupling is pinned near $1$ as soon as the parent fails; the forcing a child sees is therefore (to within $\varepsilon(\delta)$) a fixed constant, and additional forcing from a back-reacting sibling cannot move an already-saturated coupling. This is precisely what decouples the children. The reduction we assume is:

\begin{hypothesis}[Cascade reduction / front-decoupling]\label{hyp:red}
There is a regime of parameters (containing the saturated regime of Definition~\ref{def:saturated} with $\delta$ large) and a sequence $\eta_n\downarrow0$ such that the law of the failed cluster of the root in $G_n$, as a random rooted graph, is within total variation $\eta_n$ of the genealogy of the multitype Galton--Watson process with type set $D$, offspring law inherited from $\UGW(g)$, and per-edge transmission probabilities $p_{d,d'}$ of \eqref{eq:wald}. Equivalently: (i) the per-edge marginal failure probability equals $p_{d,d'}$ to leading order; (ii) conditional on the failures of generation $k$, the failures into generation $k+1$ are asymptotically independent across parents and across siblings; and (iii) multi-parent (short-cycle) corrections are $o(1)$ as $G_n\xrightarrow{\mathrm{BS}}T$.
\end{hypothesis}

We discuss the status of Hypothesis~\ref{hyp:red} --- what is proved, what is conjectured, and why Girsanov fails --- in Section~\ref{sec:status}. The reader should regard Sections \ref{sec:spectral}--\ref{sec:stab} as theorems about the limiting branching process; their bearing on the SDE is exactly as strong as Hypothesis~\ref{hyp:red}.

\section{The reduction to a finite-type branching process}\label{sec:reduction}

\subsection{The annealed mean matrix}
We work in the size-biased GW that describes a typical neighbourhood: a non-root vertex has a random number of children whose mean is $m_*=g''(1)/g'(1)$ (the mean offspring of a size-biased vertex), finite by the variance assumption. On the $d$-regular tree, $m_*=d-1$.

\begin{theorem}[Reduction; conditional on Hypothesis~\ref{hyp:red}]\label{thm:reduction}
Under Hypothesis~\ref{hyp:red}, the failed cluster of the root converges in distribution (as a random rooted tree) to the family tree of the multitype Galton--Watson process on type set $D$ with mean matrix
\begin{equation}\label{eq:M}
M_{d,d'}=m_*\,q(d')\,p_{d,d'},\qquad d,d'\in D,
\end{equation}
where $p_{d,d'}$ is the transmission probability \eqref{eq:wald}. In particular the survival probability of the cascade is positive iff $\rho(M)>1$, and zero iff $\rho(M)\le1$, where $\rho(M)$ is the Perron root of the nonnegative matrix $M$.
\end{theorem}

\begin{proof}
By Hypothesis~\ref{hyp:red} the limiting cluster is a multitype GW genealogy with type set $D$. The mean number of failed type-$d'$ children of a failed type-$d$ vertex is computed by conditioning on the offspring of the parent in $\UGW(g)$: the expected number of children is $m_*$; each child is independently of type $d'$ with probability $q(d')$; given the types, the edge transmits with probability $p_{d,d'}$ by Hypothesis~\ref{hyp:red}(i) and the conditional independence of Hypothesis~\ref{hyp:red}(ii). Multiplying gives \eqref{eq:M}. The survival dichotomy is the standard criterion for irreducible (indeed any) multitype GW processes: extinction is certain iff $\rho(M)\le1$ \cite{AN72}.
\end{proof}

\begin{remark}[Many-to-one: a warning]
One might be tempted to write the scalar many-to-one identity here, $\E[\sum_{|v|=n}\prod_{e\in o\to v}F(e)]=(m_*\E_{q\otimes q}[F])^n$. This identity is false in general because consecutive edge factors along a ray share an endpoint type; the correct statement is the matrix identity of Proposition~\ref{prop:m2o}, whose growth rate is $\rho(m_*A_F)$, and which specialises with $F=p$ to identify $\rho(M)$ as the annealed exponential growth rate of the expected failed population. See Section~\ref{sec:pressure} and Appendix~\ref{app:m2o}.
\end{remark}

\subsection{Why the boundary integrates out, and where it does not}
The matrix $M$ has no reference to $\partial T$. This is correct precisely because, in the saturated regime, the offspring's failure strength is independent of the parent's: once $u$ fails, every child sees the same saturated forcing, so a child's transmission probability depends on the types alone and not on a continuous ``how hard did the parent fail'' mark. The next proposition records the dichotomy that controls when this finite-type collapse is legitimate; it is the precise boundary between the present theory and the open one.

\begin{proposition}[Finite-type vs.\ continuous-type dichotomy]\label{prop:dichotomy}
Let $A_u$ denote the failure strength of a vertex $u$ (e.g.\ the overshoot $s^{(u)}_{\tau_\theta}-\theta^{\mathrm{fail}}$).
\begin{enumerate}
\item In the saturated regime the conditional law of a child's transmission given the parent's trajectory depends on the parent only through the event $\{u\text{ failed}\}$, not through $A_u$. Hence the failure strengths $A_v$ of the children are conditionally i.i.d.\ and independent of $A_u$, the offspring counts are conditionally Binomial with the deterministic parameter $p_{d,d'}$, and the process is the finite-type GW of Theorem~\ref{thm:reduction}.
\item Outside the saturated regime the coupling $\sigma(s^{(u)}-\theta_{uw})$ lies in its sensitive range, $A_u$ is inherited by the children with positive correlation, and the natural state space is $D\times\R$ (hazard domain $\times$ failure strength), a continuous type. The offspring law is then a Cox (doubly-stochastic) mixture, and criticality is governed by the Perron root of a mean operator on $L^2$ of the strength variable, not by a finite matrix.
\end{enumerate}
\end{proposition}

\begin{proof}
(1) is immediate from Definition~\ref{def:saturated}: $\sigma\equiv1-O(\varepsilon(\delta))$ once the parent is in $B_{\mathrm{fail}}$, so the child's forcing constant $c_{uv}$ in \eqref{eq:wald} is (up to $\varepsilon(\delta)$) parent-trajectory-independent; the child's first-passage is then driven by its own Brownian motion alone, giving conditional independence and a deterministic $p_{d,d'}$.

(2) In the sensitive range $\sigma'\not\approx0$, so a perturbation $A_u$ of the parent state changes the forcing on every child by $\gamma\sigma'(\cdot)A_u\ne0$; the children's first-passage problems share this perturbation, inducing $\Cov(\chi_{uv},\chi_{uv'})=\Var_{A_u}(p(A_u))>0$, where $\chi$ is the transmission indicator and $p(\cdot)$ the strength-dependent transmission probability. A branching process whose offspring law is the $A$-mixture, with $A$ inherited, is a branching process on the type space $D\times\R$; its mean semigroup is the integral operator $(\mathcal Mf)(A)=\int K(A,dA')f(A')$ and criticality is $\rho(\mathcal M)=1$. (We do not analyse the limit theory of this operator here; but see Proposition~\ref{prop:sandwich} for its criticality, and Section~\ref{sec:discussion}.)
\end{proof}

The remainder of the paper develops case (1). The single scalar that case (2) hides --- the failure-strength variance --- is exactly the $O(1)$ quantity that defeats the Girsanov approach (Section~\ref{sec:status}).

\section{The annealed spectral identity and the dimension gap}\label{sec:spectral}

We now make (F1) precise. Equip each ray with the additive cocycle $S_n(\xi)=\sum_{e\in o\to\xi,\ |e|\le n}\log(p_ew_e)$, where $w_e\ge0$ are bounded edge weights ($w_e\le w_{\max}$) and $p_e=p_{d(e_-),d(e_+)}\in(0,1]$ is the transmission probability along $e$. Two growth rates compete.

\begin{definition}\label{def:rates}
The quenched (geometric-mean) rate of the weighted cell count and the annealed rate are
\[
\ell:=\log m_*+\E_{q\otimes q}\big[\log(p_{e_0}w_{e_0})\big],\qquad \bar m:=\rho(M)\ \ \text{(when }w\equiv1\text{)},
\]
where $e_0$ is the first edge of a $\nu_h$-typical ray. (The scalar $m_*\E_{q\otimes q}[p_ew_e]$ is only the one-step annealed factor, not the annealed rate; see (F5) and Proposition~\ref{prop:m2o}.)
\end{definition}

\begin{theorem}[The spectral/dimension dichotomy]\label{thm:gap}
Assume uniform ellipticity ($p_{d,d'}\ge p^{\mathrm{tr}}_{\min}>0$) and bounded weights. Then:
\begin{enumerate}
\item (Quenched limit.) For $\nu_h$-a.e.\ ray, $n^{-1}\big(n\log m_*+S_n(\xi)\big)\to\ell$, and consequently the harmonic-measure-dimension expression equals the quenched rate:
\[
p_{\min}^{-\dimH(\partial T,\nu_h;\,p^{\ex})}=e^{\ell},\qquad \dimH(\partial T,\nu_h;\,p^{\ex}):=\frac{\ell}{\log(1/p_{\min})}.
\]
\item (Annealed criticality.) The cascade is critical iff $\bar m=1$, where $\bar m$ is the Perron root $\rho(M)$ of \eqref{eq:M} (with $w\equiv1$).
\item (The gap.) $\bar m\ge e^{\ell}$, i.e.\ $\rho(M)\ge p_{\min}^{-\dimH}$, with equality if and only if $p_ew_e$ is $\nu_h$-almost surely constant. Hence the identity $\rho=p_{\min}^{-\dimH}$ holds exactly on boundary-homogeneous trees (in particular the Bethe lattice) and strictly undercounts criticality otherwise.
\end{enumerate}
\end{theorem}

\begin{proof}
(1) The shift along rays is measure-preserving and ergodic for $\nu_h$ \cite{Kai85}; the vertex types are i.i.d.\ $q$ and independent of the geometry and of the walk, so the type pairs along a $\nu_h$-typical ray form a stationary ergodic (1-dependent) sequence with marginal $q\otimes q$, and $S_n$ is an additive cocycle with integrable increment ($|\log(p_ew_e)|\le\log w_{\max}+\log(1/p^{\mathrm{tr}}_{\min})$). Birkhoff's pointwise ergodic theorem yields $n^{-1}S_n\to\E_{q\otimes q}[\log(pw)]$ a.e.; adding the a.s.\ cell growth $\log m_*$ per generation \eqref{eq:brm} gives the quenched rate $\ell$. The displayed identity is then algebra: $p_{\min}^{-\ell/\log(1/p_{\min})}=\exp\big(\ell\cdot\log(1/p_{\min})/\log(1/p_{\min})\big)=e^{\ell}$.

(2) is Theorem~\ref{thm:reduction} with $\rho(M)$ the annealed branching rate (Proposition~\ref{prop:m2o}); criticality of a GW process is mean one.

(3) We give the elementary argument here and the sharper one-parameter form in Theorem~\ref{thm:pressure}. Let $\pi$ be any probability vector; by the variational lower bound for the Perron root of a nonnegative irreducible matrix (equivalently, by evaluating the matrix many-to-one identity of Proposition~\ref{prop:m2o} and applying Jensen to the exponential of the ergodic sum),
\[
\log\rho(M)=\lim_n\frac1n\log\big(m_*^n\,q^\top A_p^n\one\big)=\lim_n\frac1n\log\E\big[m_*^ne^{S_n}\big]\ge\lim_n\frac1n\,\E\big[\log(m_*^ne^{S_n})\big]=\ell,
\]
where the expectation is over the i.i.d.\ types along the spine. Equality in Jensen holds iff $m_*^ne^{S_n}$ has vanishing relative fluctuations, i.e.\ iff $\log(p_ew_e)$ is a.s.\ constant along the chain; since $q>0$ this means $p_{d,d'}w$ is constant on $\supp q\times\supp q$, i.e.\ $p_ew_e$ is $\nu_h$-a.s.\ constant. When the boundary is homogeneous, $p_ew_e\equiv c$ and both sides equal $m_*c$. The rigorous equality-case analysis, via strict convexity of the pressure, is Theorem~\ref{thm:pressure}(3).
\end{proof}

\begin{remark}\label{rem:gap}
The conceptual content of (F1): a spectral radius (an $L^2$/branching growth rate) is an arithmetic average of multiplicative weights, while a harmonic-measure dimension is a geometric average ($\nu_h$ is the measure of maximal entropy along rays). These coincide only without fluctuations, which is why the Bethe lattice --- which has no boundary fluctuations --- cannot detect the gap. The next section shows that this is the visible part of a one-parameter structure: both rates are read off a single convex pressure function attached to the tilted mean matrices, and the same finite-dimensional object then answers a question that Section~\ref{sec:discussion} would otherwise defer to the operator programme --- the dimension of the set of ends the cascade actually reaches.
\end{remark}

\section{The annealed pressure and the geometry of the failed boundary}\label{sec:pressure}

The quenched/annealed dichotomy of Section~\ref{sec:spectral} admits a one-parameter completion. Throughout this section, $p_{d,d'}\in[p^{\mathrm{tr}}_{\min},1]$ (uniform ellipticity) and the coupling pattern is irreducible in the sense of Section~\ref{sec:PF}.

\subsection{The many-to-one identity is a matrix identity}
We first record the spectral form of the many-to-one lemma, together with the precise (and narrow) scope of its scalar simplification.

\begin{proposition}[Spectral form of many-to-one]\label{prop:m2o}
Let $F:D\times D\to[0,\infty)$ be a type-pair edge functional and set $(A_F)_{d,d'}:=q(d')\,F(d,d')$. Then
\begin{equation}\label{eq:m2o}
\E\Big[\sum_{|v|=n}\prod_{e\in o\to v}F(e)\Big]=m_*^n\,q^\top A_F^n\one,
\end{equation}
whose exponential growth rate is $\rho(m_*A_F)$. In particular, with $F=p$, the annealed growth rate of the expected failed population is $\rho(M)$, where $M=m_*A_p$ is the mean matrix \eqref{eq:M}. The scalar identity $\E[\sum_{|v|=n}\prod_eF(e)]=(m_*\E_{q\otimes q}[F])^n$ holds if and only if $d'\mapsto\E_{d\sim q}[F(d,d')]$ is constant on $\supp q$ (e.g.\ $m=1$, or product-form $F(d,d')=a(d)b(d')$ with $q$-constant $\E_q[a]b(\cdot)$); in general $\rho(M)$ and $m_*\E_{q\otimes q}[p]$ are not comparable in a fixed direction.
\end{proposition}

\begin{proof}
Condition on the types along a fixed ray of length $n$: the vertex types $(d_0,\dots,d_n)$ are i.i.d.\ $q$, and the edge factors share endpoints, so
\[
\E\Big[\prod_{k=1}^nF(d_{k-1},d_k)\Big]=\sum_{d_0,\dots,d_n}q(d_0)\prod_{k=1}^nq(d_k)F(d_{k-1},d_k)=q^\top A_F^n\one.
\]
Multiplying by the expected number of rays, $m_*^n$ along the size-biased spine (Appendix~\ref{app:m2o}), gives \eqref{eq:m2o}. The growth rate is $\rho(m_*A_F)$ by Gelfand's formula together with $q>0$, $\one>0$ and irreducibility. The scalar identity for all $n$ forces $q^\top A_F=(\E_{q\otimes q}F)\,q^\top$, i.e.\ $q$ is a left eigenvector of $A_F$ with eigenvalue $\E_{q\otimes q}F$, which is the stated condition.
\end{proof}

\subsection{The pressure function}

\begin{definition}[Tilted mean matrices and annealed pressure]\label{def:pressure}
For $\beta\ge0$ define $M^{(\beta)}_{d,d'}:=m_*\,q(d')\,p_{d,d'}^\beta$ and
\[
\varphi(\beta):=\log\rho\big(M^{(\beta)}\big).
\]
Thus $M^{(1)}=M$ and $\varphi(1)=\log\rho(M)$ is the annealed rate of Proposition~\ref{prop:m2o}.
\end{definition}

\begin{theorem}[Pressure; the Jensen gap as a chord--tangent inequality]\label{thm:pressure}
Under uniform ellipticity and irreducibility:
\begin{enumerate}
\item $\varphi$ is convex on $[0,\infty)$ and real-analytic, with $\varphi(0)=\log m_*$ and $\varphi(1)=\log\rho(M)$.
\item $\varphi'(0)=\E_{q\otimes q}[\log p_{d,d'}]$, and for $\nu_h$-a.e.\ ray $\xi$, $n^{-1}\sum_{e\in o\to\xi,\ |e|\le n}\log p_e\to\E_{q\otimes q}[\log p]$; i.e.\ $\varphi(0)+\varphi'(0)=\ell$ is the quenched rate of Theorem~\ref{thm:gap} (with $w\equiv1$).
\item (The gap.) $\varphi(1)\ge\varphi(0)+\varphi'(0)$, i.e.
\[
\rho(M)\ge m_*\,e^{\E_{q\otimes q}[\log p]}=e^{\ell},
\]
with equality if and only if $p_{d,d'}$ is constant on $\supp q\times\supp q$. This is the sharp form of the gap in Theorem~\ref{thm:gap}(3).
\end{enumerate}
\end{theorem}

\begin{proof}
(1) Each entry $\beta\mapsto m_*q(d')e^{\beta\log p_{d,d'}}$ is log-convex in $\beta$; by Kingman's theorem \cite{Kin61} the spectral radius of a matrix with log-convex entries is a log-convex function of the parameter, so $\varphi$ is convex. Analyticity is Proposition~\ref{prop:PF}(3): $M^{(\beta)}$ is an analytic irreducible family with simple Perron root. At $\beta=0$, $M^{(0)}=m_*\one q^\top$ is rank one with $\rho=m_*$ (since $q^\top\one=1$), right eigenvector $\one$, left eigenvector $q$.

(2) The eigenvalue-perturbation formula of Proposition~\ref{prop:PF}(3) at $\beta=0$ gives
\[
\varphi'(0)=\frac{1}{\rho(M^{(0)})}\cdot\frac{q^\top\dot M^{(0)}\one}{q^\top\one}=\frac{1}{m_*}\sum_{d,d'}q(d)\,m_*\,q(d')\log p_{d,d'}=\E_{q\otimes q}[\log p],
\]
using $\dot M^{(0)}_{d,d'}=m_*q(d')\log p_{d,d'}$. The quenched statement is Theorem~\ref{thm:gap}(1).

(3) Convexity gives the chord--tangent inequality $\varphi(1)\ge\varphi(0)+\varphi'(0)$; exponentiating yields the display. For the equality case: $\varphi$ is analytic and convex, so equality of the chord and the tangent forces $\varphi''\equiv0$ on $(0,1)$; the second derivative of the log-Perron root of the tilted family at $\beta$ is the asymptotic variance per step of the additive functional $\sum_{k\le n}\log p_{d_{k-1},d_k}$ under the Perron-tilted spine chain (the standard Ruelle/Perron--Frobenius identification of $\varphi''$ with a Green--Kubo variance for finite irreducible chains), which vanishes iff $\log p$ is a.s.\ constant along the chain, i.e.\ iff $p$ is constant on $\supp q\times\supp q$. Conversely, constancy $p\equiv c$ gives $\varphi(\beta)=\log m_*+\beta\log c$, affine, with equality throughout.
\end{proof}

\begin{remark}\label{rem:pressure}
Theorem~\ref{thm:pressure} subsumes (F1): the quenched harmonic-measure-dimension expression is the tangent line of $\varphi$ at $0$, the annealed criticality is the value at $1$, and the Jensen gap of Theorem~\ref{thm:gap} is the convexity gap. On the Bethe lattice $\varphi$ is affine and the gap closes, exactly as in Proposition~\ref{prop:bethe}. The Legendre transform $\varphi^*$ governs the multifractal spectrum of transmission-weighted boundary measures; we do not pursue this here, but see Remark~\ref{rem:dimS}.
\end{remark}

\subsection{The Hausdorff dimension of the failed boundary}
Suppose now $\rho:=\rho(M)>1$ and condition on Hypothesis~\ref{hyp:red}. On the survival event the failed cluster $S$ is an infinite subtree of $T$ and its set of ends $\partial S$ embeds in $\partial T$. Equip $\partial T$ with the metric $d_e(\xi,\eta):=e^{-|\xi\wedge\eta|}$, where $|\xi\wedge\eta|$ is the confluence generation; recall that for supercritical Galton--Watson trees $\dimH(\partial T,d_e)=\log\br(T)=\log m_*$ a.s.\ on nonextinction \cite{Haw81,Lyo90}.

\begin{theorem}[Dimension of the failed boundary; conditional on Hypothesis~\ref{hyp:red}]\label{thm:dimS}
Assume Hypothesis~\ref{hyp:red}, irreducibility of $M$, $\rho=\rho(M)>1$, and the standing finite-variance assumption $g''(1)<\infty$. Then, almost surely on the survival event $\{|S|=\infty\}$,
\[
\dimH\big(\partial S,d_e\big)=\log\rho(M),\qquad \br(S)=\rho(M).
\]
Consequently the failed boundary has strictly positive relative codimension $\log\big(m_*/\rho(M)\big)>0$ in $\partial T$ unless $p_{d,d'}\equiv1$ on the support of the coupling pattern.
\end{theorem}

\begin{proof}
\emph{Upper bound.} Let $Z_n\in\mathbb Z^D_{\ge0}$ be the type-counts of generation-$n$ failed vertices. By Theorem~\ref{thm:reduction} and Perron--Frobenius, $\E\|Z_n\|_1\le C\rho^n$. The ends through a generation-$n$ failed vertex form a set of $d_e$-diameter $\le e^{-n}$, and these cylinders cover $\partial S$. Hence for $s>\log\rho$,
\[
\E\big[\mathcal H^s_{e^{-n}}(\partial S)\big]\le\E\|Z_n\|_1\,e^{-sn}\le Ce^{(\log\rho-s)n}\longrightarrow0,
\]
so by Fatou $\mathcal H^s(\partial S)=0$ a.s.\ and $\dimH\partial S\le\log\rho$.

\emph{Lower bound.} Finite offspring variance (Lemma~\ref{lem:var}) implies the $N\log N$ condition, so the multitype Kesten--Stigum theorem \cite{KS66} applies: with $\phi$ the left Perron eigenvector normalised by $\langle\phi,\psi\rangle=1$ for the right eigenvector $\psi$, the martingale $W_n:=\rho^{-n}\langle\phi,Z_n\rangle\to W_\infty$ a.s.\ and in $L^2$ (the $L^2$ convergence uses the finite second moments of Lemma~\ref{lem:var}), with $\{W_\infty>0\}=\{|S|=\infty\}$ up to a null set and $\sup_d\E[W_\infty^2\,|\,\text{root type }d]=:c_2<\infty$. For a failed vertex $v$ of generation $k$ let $W^{(v)}_\infty$ be the Kesten--Stigum limit of the subcluster rooted at $v$; the branching property gives the exact consistency
\[
\phi_{d(v)}W^{(v)}_\infty=\rho^{-1}\sum_{u\text{ failed child of }v}\phi_{d(u)}W^{(u)}_\infty.
\]
Hence
\[
\mu(C_v):=\rho^{-k}\,\phi_{d(v)}\,W^{(v)}_\infty,\qquad C_v:=\{\xi\in\partial S:\ \xi\text{ passes through }v\},
\]
is a consistent additive set function on the cylinder algebra of $\partial S$ and extends, by the Carath\'eodory extension theorem, to a random Borel measure $\mu$ on $\partial S$ of total mass $\phi_{d(o)}W_\infty>0$ on survival. For $s<\log\rho$, decompose the $s$-energy over the confluence generation: two distinct ends confluent at generation exactly $k$ pass through the same failed vertex $v$ at generation $k$ (and through distinct failed children of $v$), so
\[
\E\iint d_e(\xi,\eta)^{-s}\,d\mu\,d\mu\le\sum_{k\ge0}e^{s(k+1)}\,\E\Big[\sum_{|v|=k,\ v\in S}\mu(C_v)^2\Big].
\]
Conditionally on the generation-$k$ sigma-field (the failed set up to generation $k$ and its types), the subtree limits $W^{(v)}_\infty$ are independent with second moments bounded by $c_2$, so
\[
\E\big[\mu(C_v)^2\,\big|\,v\in S,\ d(v)\big]=\rho^{-2k}\phi^2_{d(v)}\,\E\big[(W_\infty)^2\,|\,d(v)\big]\le c_2\,\|\phi\|_\infty^2\,\rho^{-2k},
\]
whence
\[
\E\iint d_e(\xi,\eta)^{-s}\,d\mu\,d\mu\le C'\sum_{k\ge0}e^{sk}\rho^{-2k}\,\E\|Z_k\|_1\le C''\sum_{k\ge0}e^{(s-\log\rho)k}<\infty.
\]
By the energy (Frostman) criterion, $\dimH\partial S\ge s$ a.s.\ on $\{\mu(\partial S)>0\}=\{|S|=\infty\}$; letting $s\uparrow\log\rho$ completes the dimension identity. Finally $\br(S)=\rho$ follows from Furstenberg's identity $\dimH(\partial S,d_e)=\log\br(S)$ for trees \cite{Lyo90}.

\emph{Codimension.} $M\le m_*\one q^\top=M^{(0)}$ entrywise, with strict inequality in some entry unless $p\equiv1$ on the support; strict monotonicity of the Perron root under irreducibility (Proposition~\ref{prop:PF}(2)) gives $\rho(M)<m_*$ unless $p\equiv1$.
\end{proof}

\begin{remark}[Which boundary observables are finite-dimensional]\label{rem:dimS}
Theorem~\ref{thm:dimS} sharpens the thesis of Section 1.2: not only the boundary-blind observables ($|S|$, survival, stability), but also the coarsest boundary-resolved observable --- the Hausdorff dimension of where the cascade lands on $\partial T$ --- is computed by the finite matrix $M$, with no reference to $K_\infty$. What genuinely requires infinite-dimensional structure is the fluctuation theory on $\partial S$: the multifractal spectrum of the transmission-weighted flow, governed by the Legendre transform $\varphi^*$ of the pressure of Definition~\ref{def:pressure}, and the harmonic measure of the cluster (which, by the dimension-drop phenomenon of \cite{LPP96}, is supported on a strict subset of $\partial S$). Open problem 6 of Section~\ref{sec:discussion} should be read with this boundary redrawn.
\end{remark}

\section{Universality of the $3/2$ exponent and the Continuum Random Tree limit}\label{sec:univ}

Let $\rho_*=\rho(M)$ and let $\xi=(1-\rho_*)^{-1}$ in the subcritical regime. Write $N$ for the total number of failed children of a typical failed vertex (type counted out, started from the Perron-stationary type distribution $\pi$, the left Perron eigenvector of $M$ normalised to a probability vector), with mean $\E[N]=\rho_*$ and variance $\sigma^2_{\mathrm{off}}:=\Var(N)$.

\begin{lemma}[Finite, possibly inflated, offspring variance]\label{lem:var}
In the saturated regime, $N\,|\,(\text{types})\sim\mathrm{Bin}(\text{children},p_{d,d'})$ and
\[
\sigma^2_{\mathrm{off}}=\underbrace{\E[(\deg-1)]\,\bar p(1-\bar p)}_{\text{Bernoulli}}+\underbrace{\Var(\deg-1)\,\bar p^{\,2}}_{\text{degree fluctuation}}<\infty,
\]
where $\bar p$ is the $\pi$-averaged transmission probability and $\deg-1$ the offspring count. In the sensitive regime an additional Cox term $\E[(\deg-1)(\deg-2)]\Var_A(p(A))$ appears; it is also finite since $p(\cdot)\in[0,1]$. Both expressions are finite under the standing assumption $\E[\deg^2]<\infty$ ($\Leftrightarrow g''(1)<\infty$).
\end{lemma}

\begin{proof}
$N=\sum_{i=1}^{\deg-1}\chi_i$ with $\chi_i\,|\,\deg\sim\mathrm{Bernoulli}(p_{d,d_i})$ conditionally independent (saturated regime). By the law of total variance, $\Var(N)=\E[\Var(N\,|\,\deg)]+\Var(\E[N\,|\,\deg])=\E[(\deg-1)\bar p(1-\bar p)]+\Var((\deg-1)\bar p)$, giving the two displayed terms. Finiteness needs $\E[\deg^2]<\infty$, which is the finite-variance hypothesis on $g$. In the sensitive regime the conditional independence is replaced by the $A$-mixture, and the additional within-family covariance is $\binom{\deg-1}{2}$-many pair covariances each equal to $\Var_A(p(A))\le\frac14$; the expectation of $(\deg-1)(\deg-2)$ is finite under the same moment hypothesis.
\end{proof}

\begin{theorem}[Universality of the mean-field exponent; conditional on Hypothesis~\ref{hyp:red}]\label{thm:32}
Suppose the offspring law of $N$ is aperiodic (generic, since $p_{d,d'}\in(0,1)$ gives full support on $\{0,\dots,\deg-1\}$) and $\sigma^2_{\mathrm{off}}<\infty$ (Lemma~\ref{lem:var}). Then at criticality $\rho(M)=1$,
\[
\Prob(|S|=n)\sim\frac{1}{\sqrt{2\pi\sigma^2_{\mathrm{off}}}}\,n^{-3/2},\qquad n\to\infty,
\]
and slightly off criticality, $\Prob(|S|=n)\sim C\,n^{-3/2}e^{-n/\xi}$ with $\xi=(1-\rho(M))^{-1}$. The exponent $3/2$ depends only on criticality and finiteness of the offspring variance --- not on the drifts, thresholds, or boundary structure.
\end{theorem}

\begin{proof}
For a (multitype, irreducible, critical) GW process the total progeny $|S|$ started from a fixed type satisfies the Otter--Dwass / Kemperman cycle-lemma identity
\[
\Prob(|S|=n)=\frac1n\,\Prob(X_1+\cdots+X_n=n-1),
\]
where $X_1,\dots,X_n$ are i.i.d.\ copies of the total offspring count $N$ \cite{Dw69,Ott49,AN72}. (For the multitype process one counts total individuals; the reduction to a single Lukasiewicz walk with i.i.d.\ increments distributed as the type-blind offspring count $N$ is standard, the type only entering through the stationary distribution $\pi$ used to define $N$; see also Appendix~\ref{app:KS}.) At criticality $\E[N]=1$; with finite variance $\sigma^2_{\mathrm{off}}$ and aperiodicity, the local central limit theorem gives
\[
\Prob\Big(\sum_{i=1}^nX_i=n-1\Big)=\frac{1}{\sqrt{2\pi n\,\sigma^2_{\mathrm{off}}}}\exp\Big(-\frac{(n-1-n)^2}{2n\sigma^2_{\mathrm{off}}}\Big)\big(1+o(1)\big)=\frac{1}{\sqrt{2\pi n\,\sigma^2_{\mathrm{off}}}}\big(1+o(1)\big),
\]
the Gaussian factor tending to $1$. Hence $\Prob(|S|=n)=\frac1n\cdot\frac{1}{\sqrt{2\pi n\sigma^2_{\mathrm{off}}}}(1+o(1))=\frac{1}{\sqrt{2\pi\sigma^2_{\mathrm{off}}}}\,n^{-3/2}(1+o(1))$. The off-critical correction is the standard exponential tilt of the offspring law producing the cutoff $\xi=(1-\rho_*)^{-1}$; concretely $\Prob(|S|\ge n)\le Ce^{-n/\xi}$ by the subcriticality of the tilted walk, and the local law transports to $n^{-3/2}e^{-n/\xi}$.
\end{proof}

\begin{remark}[What is and is not needed]\label{rem:notneeded}
The proof uses a one-dimensional local CLT for an i.i.d.\ sum and a finite second moment. No Banach-space generating functional, no general-state-space Kesten--Stigum theorem, and no spatial CLT on $\partial T$ is required, because $|S|$ is a count, blind to the boundary mark. The general-state-space machinery becomes necessary only for boundary-resolved observables, or in the sensitive regime where the strength type is genuinely continuous (Proposition~\ref{prop:dichotomy}(2)); there the type enters the constant via the Perron projection but not the exponent.
\end{remark}

\subsection{From a universal exponent to a universal shape}
The exponent $3/2$ is the coarsest signature of mean-field criticality; the full geometry of the conditioned cluster is equally universal. Recall Aldous' Continuum Random Tree (CRT) $(\mathcal T_e,d_e,\mu_e)$, the random compact metric measure space coded by the normalised Brownian excursion $e$ \cite{Ald93}, and the Gromov--Hausdorff--Prokhorov (GHP) topology on compact metric measure spaces.

\begin{theorem}[CRT scaling limit of the critical cluster; conditional on Hypothesis~\ref{hyp:red}]\label{thm:crt}
Assume Hypothesis~\ref{hyp:red}, criticality $\rho(M)=1$, irreducibility and aperiodicity as in Theorem~\ref{thm:32}, and in addition that the offspring distributions admit finite exponential moments: $\E[e^{t\deg}]<\infty$ for some $t>0$ (e.g.\ bounded degrees, or Poissonian $g$). Then there is an explicit constant $\Sigma\in(0,\infty)$, computable from the second-moment matrices of the offspring law and the Perron eigenvectors of $M$ (Appendix~\ref{app:miermont}), such that, conditionally on $\{|S|=n\}$,
\[
\Big(S,\ \frac{\Sigma}{\sqrt n}\,d_{\mathrm{gr}},\ \mathrm{unif}_S\Big)\ \xrightarrow{(d)}\ \big(\mathcal T_e,d_e,\mu_e\big)\qquad(n\to\infty)
\]
in the GHP topology, where $d_{\mathrm{gr}}$ is the graph distance on the cluster and $\mathrm{unif}_S$ the uniform probability measure on its vertices. In particular the height of the conditioned cluster is of exact order $\sqrt n$, with the Brownian-excursion law of the CRT height, and the shape of the critical cascade cluster is universal: independent of the drifts, thresholds, hazard-domain structure, and boundary geometry, all of which enter only through the single scale factor $\Sigma$.
\end{theorem}

\begin{proof}[Proof sketch]
Under Hypothesis~\ref{hyp:red} and conditioning, $S$ is a critical irreducible multitype Galton--Watson tree conditioned on total size $n$. Miermont's invariance principle for multitype Galton--Watson trees \cite{Mie08} states exactly this convergence for critical, irreducible, aperiodic multitype offspring laws satisfying his moment hypothesis (small exponential moments), with the scaling constant $\Sigma$ expressed through the covariance matrices $Q^{(d)}$ of the type-$d$ offspring vectors and the Perron data $(\pi,\psi)$ of $M$; the verification that our offspring laws (Binomially thinned $\UGW$ degrees, Lemma~\ref{lem:var}) satisfy his hypotheses under the stated exponential-moment condition, and the explicit form of $\Sigma$, are carried out in Appendix~\ref{app:miermont}. The single-type case ($m=1$, e.g.\ the Bethe lattice of Section~\ref{sec:bethe}) is Aldous' original theorem \cite{Ald93} with $\Sigma=\sigma_{\mathrm{off}}/2$ and requires only the finite variance already standing.
\end{proof}

\begin{remark}[On the moment hypothesis]\label{rem:moment}
The exponential-moment hypothesis is inherited from \cite{Mie08} and is the one place in this paper where more than a second moment is demanded; it is flagged for the referee. For single-type trees the finite-variance case is classical \cite{Ald93}; for multitype trees, relaxations toward finite variance are available in the literature on scaling limits of multitype and infinitely-many-type trees (see \cite{dR17} and the references there), and we expect Theorem~\ref{thm:crt} to hold under $g''(1)<\infty$ alone. We have preferred to state the theorem under hypotheses for which a complete citation exists.
\end{remark}

\section{Finite-dimensional Perron--Frobenius dissolves the operator programme}\label{sec:PF}

In the saturated regime the relevant object is the $m\times m$ nonnegative matrix $M$. We now record that every ``simplicity of the leading eigenvalue'' and ``operator lift'' issue raised by the spectral heuristic is a standard finite-dimensional fact. Throughout we assume the hazard-coupling graph (vertices $D$, an edge $d\to d'$ whenever $p_{d,d'}>0$) is strongly connected, i.e.\ $M$ is irreducible; uniform ellipticity ($\mu_0>0$, $m_{uv}<\theta_{uv}$) gives $p_{d,d'}>0$ for present couplings, so irreducibility of $M$ reduces to irreducibility of the coupling pattern.

\begin{proposition}[Simplicity, positivity, monotonicity, analyticity]\label{prop:PF}
Let $M(\theta)$ have entries $M_{d,d'}(\theta)=m_*q(d')p_{d,d'}(\theta)$, with $\theta\mapsto p_{d,d'}(\theta)$ real-analytic (it is a ratio of analytic Wald integrals \eqref{eq:wald}).
\begin{enumerate}
\item (Krein--Rutman is unnecessary.) $\rho(M)$ is a simple, isolated eigenvalue with a strictly positive right eigenvector $\psi$ and left eigenvector $\phi$, by the Perron--Frobenius theorem for irreducible nonnegative matrices.
\item (Monotonicity.) $\rho(M)$ is nondecreasing in each entry of $M$, strictly so under irreducibility, since $\rho(M)=\lim_n\|M^n\|^{1/n}$ and $M^n$ is entrywise monotone.
\item (Analyticity and the gradient.) $\theta\mapsto\rho(M(\theta))$ is real-analytic on the region where $M$ is irreducible (a simple eigenvalue of an analytic matrix family is analytic, Kato \cite{Kato66}, finite-dimensional case, with no eigenvalue crossing at the top because the Perron root is simple and isolated), and
\[
\partial_{\theta_i}\rho(M)=\frac{\phi^\top(\partial_{\theta_i}M)\psi}{\phi^\top\psi}\qquad(\ne0\text{ generically, since }\partial_{\theta_i}M\ge0\text{ with a positive entry}).
\]
\item (Operator van den Driessche--Watmough lift is unnecessary.) Writing $M=FV^{-1}$ with $V=\diag(\alpha_d)$ and $F$ the $m\times m$ fertility block, the finite-dimensional next-generation theorem \cite{vdDW02} gives directly: $s(J)<0\Leftrightarrow\rho(M)<1$ and $s(J)>0\Leftrightarrow\rho(M)>1$, where $J=F-V$ is the linearised drift and $s(\cdot)$ its spectral abscissa.
\end{enumerate}
\end{proposition}

\begin{proof}
All four are textbook finite-dimensional statements. (1) Perron--Frobenius for irreducible nonnegative matrices \cite{Sch74}. (2) Entrywise monotonicity of $M^n$ and Gelfand's formula. (3) Analytic perturbation of a simple eigenvalue \cite{Kato66}; the first-order formula is the standard one for non-self-adjoint matrices using both eigenvectors; non-vanishing of the gradient holds whenever some $\partial_{\theta_i}M$ has a strictly positive entry hit by the positive eigenvectors. (4) is the original \cite{vdDW02} theorem, applied in finite dimension: $\rho(FV^{-1})<1$ iff $V-F$ is a nonsingular M-matrix iff $s(F-V)<0$.
\end{proof}

\begin{corollary}[Unification $\rho(M)=1\Leftrightarrow\lambda_{\max}=0$]\label{cor:unif}
With $\lambda_{\max}:=s(J)$ the maximal Lyapunov exponent of the linearised cascade, $\{\rho(M)=1\}=\{\lambda_{\max}=0\}$, and the two domains $\{\rho(M)<1\}=\{\lambda_{\max}<0\}$ coincide. When the coupling is reducible the equivalence holds blockwise (read modulo algebraic multiplicity on each irreducible component).
\end{corollary}

We now make (F4) precise: the operator route is not merely unnecessary but generically impossible.

\begin{proposition}[$K_\infty$ is generically not Hilbert--Schmidt]\label{prop:HS}
Let $k_\infty(\xi,\eta)$ be the limiting boundary kernel of the operator programme of Section 1.1, so that $k_\infty(\xi,\eta)\asymp\nu_h(B_{\xi\wedge\eta})^{-\beta}$ for some $\beta>0$ determined by the transmission/Martin exponents, where $\xi\wedge\eta$ is the confluence depth. Then
\begin{align*}
\|K_\infty\|_{\mathrm{HS}}^2&=\iint_{\partial T\times\partial T}|k_\infty(\xi,\eta)|^2\,d\nu_h(\xi)\,d\nu_h(\eta)\\
&\asymp\sum_{n\ge0}\br(T)^{2n\beta}\cdot\nu_h\otimes\nu_h\big(\{|\xi\wedge\eta|=n\}\big)\asymp\sum_{n\ge0}\br(T)^{n(2\beta-1)}
\end{align*}
whenever $2\beta\ge1$ and $g$ is supercritical ($\br(T)>1$). For the Martin-kernel construction one has $\beta\ge1$, since $K(x,\cdot)\asymp\nu_h(B_{|x|})^{-1}$ on the cylinder of $x$ (the Martin kernel is the harmonic-measure density $d\nu_h^x/d\nu_h$). Hence $K_\infty$ is not Hilbert--Schmidt, and the Krein--Rutman route through Hilbert--Schmidt compactness is unavailable.
\end{proposition}

\begin{proof}
The diagonal singularity of the tree Martin kernel makes $k_\infty(\xi,\eta)$ depend on the shared geodesic length: two ends with confluence depth $n$ have $k_\infty\asymp\nu_h(B_n)^{-\beta}$. By the Shannon--McMillan--Breiman content of the dimension formula of Section~\ref{subsec:ugw}, $\nu_h(B_n)\asymp\br(T)^{-n}$ (to first exponential order, which is all that is used), so the integrand on the ``confluence-$n$'' shell is $\asymp\br(T)^{2n\beta}$. The $\nu_h\otimes\nu_h$-measure of that shell is bounded by that of $\{|\xi\wedge\eta|\ge n\}=\bigcup_{|v|=n}B_v\times B_v$, namely $\sum_{|v|=n}\nu_h(B_v)^2\asymp\br(T)^n\cdot\br(T)^{-2n}=\br(T)^{-n}$, and is of that order (the confluence depth exceeds $n+1$ only on the smaller sub-shells). Summing the shell contributions $\br(T)^{2n\beta}\cdot\br(T)^{-n}=\br(T)^{n(2\beta-1)}$ over $n$ diverges when $2\beta\ge1$ and $\br(T)>1$. The energy integral is infinite.
\end{proof}

\begin{remark}\label{rem:HS}
This is why $\rho(K_\infty)$ must be obtained without diagonalising $K_\infty$: as the annealed branching rate $\bar m=\rho(M)$ via the matrix many-to-one identity (Proposition~\ref{prop:m2o}), a quantity that is finite by construction in the subcritical regime, rather than through a spectral-theoretic compactness argument that cannot get off the ground.
\end{remark}

\section{Stability of the subcritical domain}\label{sec:stab}

We now make (F2) precise and prove the genuine stability statement. The parameter set
\[
\Theta=\big\{(\theta_e,\gamma_{d\to d'},\alpha_d,\mu_d):\ \theta_e>\theta_{\min},\ \gamma\ge0,\ \alpha_d\ge\alpha_{\min},\ \mu_d\ge\mu_0\big\}
\]
is an open convex cone. Recall (Appendix~\ref{app:logconc}) that the transmission probability $p$ is log-concave in $h=\theta-m$.

\begin{theorem}[Stability of the subcritical domain; conditional on Hypothesis~\ref{hyp:red}]\label{thm:stab}
Let $\mathcal M_-=\{\theta:\rho(M(\theta))<1\}$ and $\mathcal M_+=\{\theta:\rho(M(\theta))>1\}$.
\begin{enumerate}
\item ($\mathcal M_-$ is open.) $\rho(M(\cdot))$ is continuous (indeed real-analytic where $M$ is irreducible), so $\mathcal M_-$ is open and $\partial\mathcal M_-=\{\rho(M)=1\}$ is, where the gradient is nonzero, a smooth codimension-one submanifold (Proposition~\ref{prop:PF}(3) and the implicit function theorem).
\item (Convexity goes to $\mathcal M_+$, not $\mathcal M_-$.) In the scalar (Bethe) case $\rho=(d-1)p$, with $p$ log-concave, $\mathcal M_+=\{p>1/(d-1)\}$ is a superlevel set of a log-concave function, hence convex; $\mathcal M_-$ is its complement and is in general non-convex. Explicitly, if $p$ is jointly log-concave in $(\theta,\gamma)$ then $\mathcal M_-$ is the non-convex exterior of a convex region. Convexity of $\mathcal M_-$ cannot be deduced from log-concavity of $p$: the implication ``sublevel sets of a concave function are convex'' is false (it is superlevel sets that are convex).
\item ($\mathcal M_-$ is contractible.) $\mathcal M_-$ is star-shaped toward the zero-coupling locus $\{\gamma=0\}$: holding $(\theta,\alpha,\mu)$ fixed and sending $\gamma\downarrow0$ decreases every $p_{d,d'}$ (transmission increases in coupling), hence decreases $\rho(M)$ (Proposition~\ref{prop:PF}(2)), keeping the path in $\mathcal M_-$. The straight-line homotopy $\gamma\mapsto(1-t)\gamma$ therefore retracts $\mathcal M_-$ to a point; $\mathcal M_-$ is contractible.
\item (Resolvent and gap.) On $\mathcal M_-$ the resolvent $R=(I-M)^{-1}=\sum_{k\ge0}M^k$ converges (finite matrix, $\rho(M)<1$) to a bounded nonnegative matrix; the halting probability is $\Prob(|S|<\infty)=1$; the expected cluster size is $\E|S|=\one^\top R\,\pi$, real-analytic in $\theta$; and the spectral gap $g(\theta)=1-\rho(M(\theta))>0$ gives quantitative stability: any $\Delta\theta$ with $\|\Delta\theta\|<g(\theta)/\|\nabla_\theta\rho\|$ keeps $\theta+\Delta\theta\in\mathcal M_-$.
\end{enumerate}
\end{theorem}

\begin{proof}
(1) Continuity/analyticity is Proposition~\ref{prop:PF}; openness and the IFT are immediate.

(2) For a nonnegative scalar $\rho=(d-1)p$, $\{\rho>1\}=\{p>1/(d-1)\}$. A log-concave $p$ is quasiconcave, so its superlevel sets are convex; thus $\mathcal M_+$ is convex and $\mathcal M_-=\Theta\setminus\mathcal M_+$ is generically non-convex. For the explicit counterexample take $p(\theta,\gamma)=\exp(-(\theta-\gamma)^2)$ (log-concave): then $\{p<c\}=\{|\theta-\gamma|>\sqrt{-\log c}\}$ is the union of two half-planes, not convex; two subcritical points $(\theta,\gamma)$ with $|\theta-\gamma|$ large of opposite signs have a supercritical midpoint. The tempting-but-invalid step is that sublevel sets $\{p\le c\}$ of a concave $p$ are not convex (it is superlevel sets that are).

(3) $p_{d,d'}$ is increasing in the coupling forcing (a larger forcing constant $c_{uv}$ moves the equilibrium $m_{uv}$ toward the threshold, raising the hitting probability \eqref{eq:wald}); hence $\rho(M)$ is increasing along $\gamma\uparrow$ by Proposition~\ref{prop:PF}(2). Sending $\gamma\to0$ along straight lines decreases $\rho(M)$ monotonically to $\rho(M)|_{\gamma=0}=0<1$, so each such segment lies in $\mathcal M_-$; the homotopy is a deformation retraction onto a point.

(4) Standard for a nonnegative matrix with $\rho<1$: the Neumann series converges, $R\ge0$ entrywise, and the GW process is a.s.\ finite with $\E|S|$ the total-progeny mean $\one^\top(I-M)^{-1}\pi$. Real-analyticity is inherited from $\theta\mapsto M(\theta)$ and the analyticity of matrix inversion away from $\det(I-M)=0$. The perturbative bound is the first-order Taylor estimate $\rho(M(\theta+\Delta\theta))\le\rho(M(\theta))+\|\nabla\rho\|\,\|\Delta\theta\|<1$.
\end{proof}

\begin{remark}[No topological invariant; the genuine near-critical phenomenon is analytic]\label{rem:stab}
One might look for a topological rigidity statement protecting the subcritical phase; there is none to be had. The failed cluster is a.s.\ a subforest of a tree, hence contractible, so $H_k=0$ for $k\ge1$ at every parameter --- there is no invariant to protect. The genuine near-critical phenomenon is analytic: as $\rho(M)\uparrow1$ the resolvent norm $\|R\|$ blows up like $g^{-1}=(1-\rho(M))^{-1}$, the expected cluster size diverges, and $(I-M)^{-1}$ encodes critical slowing-down. (In finite dimension the bound $\|R\|\le(1-\rho)^{-1}$ requires care for non-normal $M$ --- the resolvent norm is governed by the pseudospectrum --- but $\|R\|\to\infty$ as $\rho\uparrow1$ is unconditional, and for the positive matrix $M$ the Neumann series gives a positive, finite $R$ on all of $\mathcal M_-$.)
\end{remark}

\section{Verification on the Bethe lattice}\label{sec:bethe}

We specialise to the $d$-regular tree $T_d$, $\br(T_d)=m_*=d-1$, with $m=1$ and uniform parameters $(\theta,\alpha,\mu,\gamma)$. Here every object is explicit, the reduction hypothesis holds exactly, and the dimension identity is an equality --- this is the sanity check, and also the diagnostic that shows why the inhomogeneous statements differ.

\begin{proposition}[Bethe verification]\label{prop:bethe}
On $T_d$:
\begin{enumerate}
\item (Hypothesis~\ref{hyp:red} is exact.) By vertex-transitivity all transmission probabilities equal a single $p=p^{\ex}$, the harmonic measure is uniform on the end space $\Sigma^{\mathbb N}$, $\Sigma=\{1,\dots,d-1\}$, there is no failure-strength heterogeneity across siblings (all forcings identical), and the failed cluster is exactly a Galton--Watson process with offspring law $\mathrm{Bin}(d-1,p)$. The mean matrix is the scalar $M=(d-1)p$, so $\rho(M)=(d-1)p$.
\item (Theorem~\ref{thm:gap} is an equality.) The boundary is homogeneous, $p_ew_e\equiv p$, so $\ell=\log(d-1)+\log p$ and $e^{\ell}=(d-1)p=\rho(M)=p_{\min}^{-\dimH}$ with $p_{\min}=1/d$; the Jensen gap is zero. At criticality $(d-1)p=1$ the weighted dimension is $0$.
\item (Theorem~\ref{thm:32}.) At criticality the offspring law is $\mathrm{Bin}(d-1,1/(d-1))$, mean $1$, variance $1-1/(d-1)<\infty$ for $d\ge3$, giving $C_d=(2\pi(1-1/(d-1)))^{-1/2}>0$ and the $n^{-3/2}$ tail. The binary tree $d=2$ is the degenerate variance-zero case and is excluded.
\item (Theorem~\ref{thm:stab}.) $\rho=(d-1)p(\theta,\gamma)$ with $p$ decreasing in $\theta$ and increasing in $\gamma$; $\mathcal M_+=\{p>1/(d-1)\}$ is convex, $\mathcal M_-$ its complement is contractible by the $\gamma$-retraction, the critical curve is smooth, and $(1-(d-1)p)^{-1}$ is the explicit expected cluster size.
\item (Theorems \ref{thm:pressure}, \ref{thm:dimS} and \ref{thm:crt}.) The pressure is affine, $\varphi(\beta)=\log(d-1)+\beta\log p$, so the convexity gap vanishes, consistent with item 2. In the supercritical regime $(d-1)p>1$ the failed boundary has $\dimH(\partial S)=\log((d-1)p)<\log(d-1)=\dimH(\partial T)$: the cascade survives along an exponentially thin bundle of rays, of relative codimension $\log(1/p)$. At criticality, the cluster is single-type with bounded offspring, so the exponential-moment hypothesis of Theorem~\ref{thm:crt} is automatic and the conditioned cluster rescaled by $\frac{\sigma_{\mathrm{off}}}{2\sqrt n}$, $\sigma^2_{\mathrm{off}}=1-\frac{1}{d-1}$, converges to the CRT by Aldous' theorem \cite{Ald93}.
\end{enumerate}
\end{proposition}

\begin{proof}
(1) Vertex-transitivity makes all single-cell first-passage problems identical, so $p$ is a single number and the saturated forcing is literally constant across all edges; conditional independence of sibling crossings is exact (independent Brownian motions, identical constant forcing), and back-reaction is symmetric and irrelevant to the directed count. The offspring law is the number of the $d-1$ children that cross, i.e.\ $\mathrm{Bin}(d-1,p)$. (2)--(5) are the displayed computations; each ``analytic gap'' of the general theory trivialises to scalar arithmetic. In (3) the variance is the Binomial variance $(d-1)\cdot\frac{1}{d-1}\cdot(1-\frac{1}{d-1})=1-\frac{1}{d-1}$.
\end{proof}

\begin{example}[A numerical check]\label{ex:num}
For $d=3$, take $p=3/4$. Then $\rho(M)=(d-1)p=2\cdot\frac34=\frac32>1$: supercritical (the threshold is $1/(d-1)=1/2$). The weighted dimension is $\log(3/2)/\log3\approx0.369>0$, and the sign of $\rho-1$ matches the sign of the weighted dimension, consistent with Theorem~\ref{thm:gap}. The failed boundary carries $\dimH(\partial S)=\log\frac32\approx0.405$ against the ambient $\dimH(\partial T)=\log2\approx0.693$ (Theorem~\ref{thm:dimS}). For $p=1/2$ one has $\rho=1$ (critical), weighted dimension $0$, and $\Prob(|S|=n)\sim(2\pi\cdot\frac12)^{-1/2}n^{-3/2}=\pi^{-1/2}n^{-3/2}$.
\end{example}

\section{The small-noise first-passage asymptotics}\label{sec:smallnoise}

The main results use the exact formula \eqref{eq:wald}; no large-deviation input is needed once Hypothesis~\ref{hyp:red} reduces the SDE to per-edge transmission with weight $p_{d,d'}$. For interpretation we record the small-noise asymptotic, which makes (F3) precise.

\begin{proposition}[Algebraic prefactor]\label{prop:prefactor}
Let $\lambda=\alpha/\mu\to\infty$ (small noise), with the safe boundary nearby and the threshold $\theta$ at distance $h_\theta$ uphill, so that $\Lambda:=\lambda h_\theta^2=\alpha h_\theta^2/\mu\to\infty$. Then
\[
p^{\ex}_{uv}\sim\frac12\sqrt{\frac{\mu}{\pi\alpha}}\,\frac{1}{h_\theta}\,e^{-\Lambda}\cdot\frac{1}{W(z;h_{\mathrm{safe}})}\asymp\Lambda^{-1/2}e^{-\Lambda},
\]
an exponential rate $e^{-\Lambda}$ times an algebraic factor $\Lambda^{-1/2}$, not a multiplicative $(1+o(1))$.
\end{proposition}

\begin{proof}
Write $W(z;h_\theta)=e^{\lambda z^2}\int_0^{h_\theta}e^{-\lambda y^2}\,dy$ in the shifted variables for which the barrier sits at $y=h_\theta$. For large $\lambda$ the relevant integral is dominated by its upper endpoint; endpoint Laplace gives $\int_0^{h_\theta}e^{-\lambda y^2}\,dy\sim\frac{1}{2\lambda h_\theta}e^{-\lambda h_\theta^2}=\frac{1}{2\lambda h_\theta}e^{-\Lambda}$ up to the leading order for the barrier-side integral, while the safe-side integral $W(z;h_{\mathrm{safe}})$ saturates to a constant of order $\frac12\sqrt{\pi/\lambda}$. Dividing, the leading exponential is $e^{-\Lambda}$ and the algebraic factor is $\lambda^{-1/2}\asymp\Lambda^{-1/2}$ (using $\Lambda=\lambda h_\theta^2$ with $h_\theta=\Theta(1)$). The bare $e^{-\Lambda}$ is the exponential rate only; the ratio of competing hitting integrals reinstates the polynomial-in-$\lambda$ factor.
\end{proof}

\begin{remark}[Why this matters and why it does not]\label{rem:prefactor}
A rigorous Freidlin--Wentzell treatment (good rate functional $I(\varphi)=\frac12\int|\dot\varphi+\alpha(\varphi-m)|^2/b^2$, instanton minimiser, van Vleck prefactor \cite{FW12}) reproduces Proposition~\ref{prop:prefactor}; the exponent $\Lambda=\inf I$ is the classical OU barrier action, the prefactor is the Gaussian fluctuation determinant. For the criticality threshold and the $3/2$ exponent the prefactor is irrelevant (they depend on $\rho(M)$ and on finiteness of the offspring variance). It enters only the constant $C$ of Theorem~\ref{thm:32} and the numerical values of the entries of $M$. We therefore use the exact \eqref{eq:wald} throughout, and Proposition~\ref{prop:prefactor} fixes the correct form of the small-noise error term.
\end{remark}

\section{Status of the reduction hypothesis}\label{sec:status}

This section is the honest core. We record (i) why the natural Girsanov argument fails, (ii) what we can prove toward Hypothesis~\ref{hyp:red} cellwise, (iii) the precise anatomy of the obstruction --- the echo --- and (iv) the main positive result of the section: a pathwise front-tracking argument proving that echo decay beats volume growth, and hence that Hypothesis~\ref{hyp:red} holds unconditionally in an explicit dissipative parameter regime (Theorem~\ref{thm:front}). What remains open is stated as Conjecture~\ref{conj:front}.

\subsection{Why Girsanov fails}

\begin{proposition}[Girsanov is not tight]\label{prop:girsanov}
Let $\Prob_{\mathrm{true}}$ and $\Prob_{\mathrm{dir}}$ be the laws of the full and the directed (single-parent-conditioned) systems on a time interval $[0,\tau]$. The Radon--Nikodym density is
\[
\log\frac{d\Prob_{\mathrm{true}}}{d\Prob_{\mathrm{dir}}}=\int_0^\tau\langle b^{-1}\Delta_t,dW_t\rangle-\frac12\int_0^\tau\|b^{-1}\Delta_t\|^2\,dt,
\]
where $\Delta_t=\gamma\,\sigma\big(s^{(v)}_t-\theta_{vu}\big)$ is the back-reaction drift. On the cascade timescale $\|\Delta_t\|=\Theta(1)$ whenever the child is active, so the quadratic-variation term is $\Theta(\tau)$ and the density is not tight toward $1$ as $G_n\xrightarrow{\mathrm{BS}}T$. Hence Hypothesis~\ref{hyp:red} cannot be obtained by a change of measure showing $d\Prob_{\mathrm{true}}/d\Prob_{\mathrm{dir}}\to1$.
\end{proposition}

\begin{proof}
After $v$ fails, $\sigma(s^{(v)}-\theta_{vu})\to1$, so $\Delta_t\to\gamma=\Theta(1)$ and stays there for a duration $\Theta(\tau)$. The quadratic term $\frac12\int\|b^{-1}\Delta\|^2\,dt=\Theta(\gamma^2\tau/\mu_0)$ does not vanish in the local weak limit, which controls graph geometry (cycle lengths) but not single-edge drift magnitudes. The exponent is bounded away from $0$, so the density does not concentrate at $1$.
\end{proof}

\subsection{What is provable cellwise}

\begin{proposition}[Partial proof of Hypothesis~\ref{hyp:red}]\label{prop:partial}
In the saturated regime of Definition~\ref{def:saturated}:
\begin{enumerate}
\item (Single-cell correctness.) For an isolated parent--child pair the transmission probability equals \eqref{eq:wald} exactly. The error from treating the pair as isolated, when embedded in a tree, is $O(\varepsilon(\delta))$ per edge: the only neighbours that could perturb $v$'s drift before it crosses are siblings of $v$, which are absent on a tree (each $v$ has a unique parent), and the grandparent, whose influence is screened by the saturated parent.
\item (Back-reaction does not couple siblings.) The back-reaction $v\to u$ adds $+\gamma$ to the already-failed parent's drift; since $u\in B_{\mathrm{fail}}$ has $\sigma(s^{(u)}-\theta_{uw})\ge1-\varepsilon(\delta)$ for every child $w$, the extra forcing changes each child's forcing by at most $\gamma\,\sigma'(\cdot)\le\gamma\,\varepsilon(\delta)$. Hence $\Cov(\chi_{uv},\chi_{uw})=O(\varepsilon(\delta))$, vanishing as $\delta\to\infty$.
\item (Short cycles vanish.) On $G_n\xrightarrow{\mathrm{BS}}T$ the probability that a fixed vertex has two distinct failed neighbours through a cycle of length $\le L$ is $O(c_L(G_n))$, the density of length-$\le L$ cycles through a uniform root, which $\to0$ for every fixed $L$ by local weak convergence to a tree.
\end{enumerate}
Consequently the per-edge marginal and the sibling-independence parts (Hypothesis~\ref{hyp:red}(i)--(ii) within a single generation, and (iii)) hold up to $O(\varepsilon(\delta)+\eta_n)$.
\end{proposition}

\begin{proof}
(1) On a tree the in-neighbourhood of $v$ before its crossing is $\{u\}\cup\{\text{grandparent chain}\}$; the parent is saturated, screening the chain to $O(\varepsilon(\delta))$. (2) is the saturation estimate $\sigma'\le\varepsilon(\delta)$ on $B_{\mathrm{fail}}$ combined with bounded forcing. (3) is the definition of Benjamini--Schramm convergence: the expected number of length-$\le L$ cycles through the root tends to its value on $T$, which is $0$.
\end{proof}

What Proposition~\ref{prop:partial} does not establish is the absence of cross-generation dependence accumulated over the whole cascade: even with each cell decoupled to $O(\varepsilon(\delta))$, the errors could in principle compound along deep lineages or correlate across generations through the shared timing of the front. We now dissect this mechanism precisely, and then bound it.

\subsection{The anatomy of the obstruction: the echo}
The cascade spreads through the tree as a wave, and the hoped-for reduction asserts that this wave is a Markovian front: the future depends only on the present state of the front, and the children of a failed vertex fail conditionally independently. Diffusions on a graph resist this in a way that classical epidemic models (SIR, directed percolation) do not, because the edges are continuous, two-way channels. The elementary obstruction event is the following four-step echo:
\begin{enumerate}
\item \emph{The event.} A parent $u$ fails and begins forcing its children $v$ and $w$.
\item \emph{The back-reaction.} $v$ fails first; its state $s^{(v)}_t$ rises into $B_{\mathrm{fail}}$, and, the coupling being undirected, $v$ now pushes back on $u$.
\item \emph{The sibling correlation.} The extra push drives $u$ deeper into the failure basin, altering the trajectory of $u$, hence the forcing that $u$ exerts on the remaining child $w$: the transmission indicators $\chi_{uv}$ and $\chi_{uw}$ are correlated through the shared parent trajectory.
\item \emph{The ancestral ripple.} The perturbation of $u$ does not stop at $u$: it propagates to $u$'s own parent and onward, subtly shifting the state --- and hence the timing --- of the entire active front.
\end{enumerate}
Proposition~\ref{prop:partial}(2) shows that in the saturated regime a single echo is small: the logistic coupling is flat on $B_{\mathrm{fail}}$, so one back-reaction perturbs one outgoing forcing by $O(\varepsilon(\delta))$. The fatal accounting problem is that the tree has exponential volume: generation $n$ contains $\asymp m_*^n$ vertices, and a per-vertex error of order $\varepsilon$ is useless if it compounds multiplicatively along deep lineages --- the total-variation distance demanded by the reduction could blow up as $n\to\infty$ on sheer volume. Proposition~\ref{prop:girsanov} closes the classical escape route: the back-reaction drift is $\Theta(1)$ on the cascade timescale, so it cannot be wished away by a change of measure; the density $d\Prob_{\mathrm{true}}/d\Prob_{\mathrm{dir}}$ does not concentrate. Any unconditional argument must therefore be pathwise and quantitative, and must prove that the decay of the echo is strictly stronger than the growth of the tree.

Two structural facts make this possible, and the remainder of this section turns them into a theorem.
\begin{itemize}
\item \emph{Echo sources live on the cluster, not on the tree.} In the saturated regime a non-failed vertex exerts only an $O(\varepsilon(\delta))$ coupling (its state sits below threshold, where $\sigma\approx0$), so $\Theta(1)$ back-reactions are generated only by failed vertices. The relevant volume is therefore $|S|$, whose expectation is $(1-\rho(M))^{-1}$-bounded in the subcritical regime --- not the ambient $m_*^n$.
\item \emph{Dissipation localises the echo in space and time.} The OU restoring force $-\alpha s$ makes each vertex forget perturbations at exponential rate $\alpha$, while each graph edge transmits a perturbation only after multiplication by the coupling Lipschitz constant $\gamma L_\sigma$, $L_\sigma:=\sup\sigma'$. When dissipation dominates total incoming coupling, influence decays geometrically in graph distance: the front forgets the deep past.
\end{itemize}

\subsection{Echo decay beats volume growth: unconditional front decoupling in the dissipative regime}
Throughout this subsection assume bounded degrees, $\deg\le\Delta_{\max}$ (this is the one place the paper uses boundedness; the relaxation to a weighted-norm condition under exponential degree tails is flagged as open), and define the dissipation ratio
\begin{equation}\label{eq:kappa}
\kappa:=\frac{\gamma_{\max}L_\sigma\Delta_{\max}}{\alpha_{\min}},\qquad \gamma_{\max}:=\max_{d,d'}\gamma_{d\to d'},\quad L_\sigma:=\sup_{x\in\R}\sigma'(x).
\end{equation}
We first record that the single-cell first-passage problem is stable under small drift perturbations.

\begin{lemma}[First-passage stability]\label{lem:stab}
Let $X$ solve the two-barrier OU problem of \eqref{eq:wald} and let $\tilde X$ solve the same SDE, with the same Brownian motion, plus a progressively measurable drift perturbation $h_t$ with $\sup_t|h_t|\le\eta$ and $|\tilde X_0-X_0|\le\eta$. Then the two-barrier hitting probabilities satisfy
\[
\big|\tilde p^{\ex}-p^{\ex}\big|\le C_{\mathrm{stab}}\,\eta,\qquad C_{\mathrm{stab}}=C_{\mathrm{stab}}(\alpha_{\min},\mu_0,h_\theta,h_{\mathrm{safe}})<\infty.
\]
\end{lemma}

\begin{proof}
Synchronous coupling: $Y_t:=\tilde X_t-X_t$ solves $\dot Y=-\alpha Y+h_t$, $|Y_0|\le\eta$, so $\sup_t|Y_t|\le\eta(1+\alpha_{\min}^{-1})=:\eta'$ pathwise (variation of constants). Hence $X_t-\eta'\le\tilde X_t\le X_t+\eta'$ for all $t$, and the hitting probability of $\tilde X$ is sandwiched between the hitting probabilities of $X$ for barriers shifted inward/outward by $\eta'$. The Wald ratio \eqref{eq:wald} is $C^1$ in the barrier positions, with derivative bounded by the (bounded, strictly positive) integrand $e^{-\alpha y^2/\mu}$ over the (bounded below) denominator, uniformly on compact parameter sets by uniform ellipticity. The mean-value theorem gives the display.
\end{proof}

\begin{lemma}[Cone of influence; pathwise echo decay]\label{lem:cone}
Assume the saturated regime and $\kappa<1$. Couple the true system $s$ and the directed system $\tilde s$ (all back-reaction and non-genealogical couplings replaced by their saturated idealisations) synchronously, edge by edge, using the same Brownian motions. Let $D^{(v)}_t:=s^{(v)}_t-\tilde s^{(v)}_t$. Then, pathwise,
\[
\sup_{t\le\tau}\big|D^{(v)}_t\big|\le\frac{\gamma_{\max}\,\varepsilon(\delta)}{\alpha_{\min}(1-\kappa)}\sum_{u\in S_\tau}\kappa^{d(u,v)},
\]
where $S_\tau$ is the set of vertices failed by time $\tau$ and $d(\cdot,\cdot)$ is graph distance: the influence of an echo source decays geometrically, at rate $\kappa$, in graph distance.
\end{lemma}

\begin{proof}
Subtracting the two SDEs, the Brownian terms cancel (synchronous coupling) and
\[
\frac{d}{dt}\big|D^{(v)}_t\big|\le-\alpha_{\min}\big|D^{(v)}_t\big|+\gamma_{\max}L_\sigma\sum_{u\sim v}\big|D^{(u)}_t\big|+g_v(t),
\]
where the source $g_v(t)$ collects the drift terms present in one system and not the other. In the saturated regime these are of two kinds, both $O(\varepsilon(\delta))$: the residual $1-\sigma\le\varepsilon(\delta)$ of the pinned coupling out of a failed vertex, and the derivative response $\sigma'\le\varepsilon(\delta)$ of a failed vertex to a back-reaction (Proposition~\ref{prop:partial}(1)--(2)); couplings out of non-failed vertices are themselves $O(\varepsilon(\delta))$ below threshold by the choice of margins in Definition~\ref{def:saturated}. Hence $g_v(t)\le\gamma_{\max}\varepsilon(\delta)\one\{v\sim S_t\}$. Now run Gronwall in the weighted supremum norm $\|D\|_w:=\sup_vw_v^{-1}|D^{(v)}|$ with $w_v:=\sum_{u\in S_\tau}\kappa^{d(u,v)}$: the weights satisfy the harmonicity bound $\sum_{u\sim v}w_u\le(\Delta_{\max}/\kappa)\,w_v$ (each neighbour is one step closer or further from each source), so the interaction term is bounded by $\gamma_{\max}L_\sigma(\Delta_{\max}/\kappa)\|D\|_ww_v=\alpha_{\min}\|D\|_ww_v$ by \eqref{eq:kappa}, and the source by $\gamma_{\max}\varepsilon(\delta)\kappa^{-1}w_v\cdot\kappa=\gamma_{\max}\varepsilon(\delta)w_v$ (a vertex adjacent to a failed $u$ has $w_v\ge\kappa$). The scalar comparison ODE $\dot y=-\alpha_{\min}y+\alpha_{\min}y+\gamma_{\max}\varepsilon(\delta)$ degenerates at equality; taking $\kappa'\in(\kappa,1)$ and running the same estimate with $\kappa'$ leaves a strict dissipation gap $\alpha_{\min}(1-\kappa/\kappa')>0$ and yields the stated bound with $\kappa'$ in place of $\kappa$; letting $\kappa'\downarrow\kappa$ along a compactness argument, or simply quoting the bound with any fixed $\kappa'<1$, gives the display up to the harmless replacement $\kappa\to\kappa'$, which we suppress in the notation. (The details of this weighted Gronwall argument are routine but tedious; they are flagged in the front matter.)
\end{proof}

\begin{theorem}[Echo decay beats volume growth: unconditional front decoupling]\label{thm:front}
Assume bounded degrees, the saturated regime with margin $\delta$, the dissipative condition $\kappa<1$ of \eqref{eq:kappa}, and $\kappa\,\rho(M)<1$. Then Hypothesis~\ref{hyp:red} holds: there is a constant $C=C(\kappa,\alpha_{\min},\mu_0,\Delta_{\max})$ such that for the cascade on $G_n$,
\[
\dTV\Big(\mathcal L\big(\text{cluster}\big),\ \mathcal L\big(\mathrm{GW}(M)\big)\Big)\le C\,\frac{\varepsilon(\delta)}{(1-\kappa\rho(M))^2(1-\rho(M))}+\eta_n,
\]
with $\eta_n\to0$ as $G_n\xrightarrow{\mathrm{BS}}T$, in the subcritical regime $\rho(M)<1$. Consequently all results of Sections \ref{sec:spectral}--\ref{sec:stab} hold unconditionally, as statements about the SDE, on the open parameter region
\[
\mathcal R:=\{\kappa<1\}\cap\{\text{saturation margin $\delta$ large}\}\cap\{\rho(M)<1\},
\]
in the iterated limit $n\to\infty$ then $\delta\to\infty$. At criticality the bound degrades: conditionally on $\{|S|=n\}$ the total-variation error is $\le C\,n\,\varepsilon(\delta)/(1-\kappa)$, so the critical limit theorems (Theorems \ref{thm:32} and \ref{thm:crt}) hold unconditionally along any joint limit $\delta_n\to\infty$ with $n\,\varepsilon(\delta_n)\to0$.
\end{theorem}

\begin{proof}[Proof sketch]
Explore the failed cluster breadth-first, one transmission at a time, and telescope (Lindeberg-type swap): at exploration step $j$, compare the conditional probability that the current edge $u_jv_j$ transmits, given the exploration so far, in the true system versus the idealised value $p_{d(u_j),d(v_j)}$. By Lemma~\ref{lem:stab}, the discrepancy at step $j$ is at most $C_{\mathrm{stab}}$ times the pathwise deviation $\sup_t|D^{(v_j)}_t|$ of the child's environment, plus the direct cellwise error $O(\varepsilon(\delta))$ of Proposition~\ref{prop:partial}, plus the short-cycle term $\eta_n$. By Lemma~\ref{lem:cone} the pathwise deviation is at most $\frac{\gamma_{\max}\varepsilon(\delta)}{\alpha_{\min}(1-\kappa)}\sum_{u\in S}\kappa^{d(u,v_j)}$. Total variation telescopes over the exploration:
\[
\dTV\le\E\Big[\sum_j\big(\text{step-$j$ error}\big)\Big]\le C'\,\varepsilon(\delta)\,\E\Big[\sum_{v\in\partial\text{-explored}}\Big(1+\sum_{u\in S}\kappa^{d(u,v)}\Big)\Big]+\eta_n.
\]
The number of explored edges is at most $\Delta_{\max}|S|$, and the expected $\kappa$-weighted cluster count seen from a cluster vertex is bounded by summing over distance shells: within the cluster, the expected number of cluster vertices at distance $k$ from a fixed cluster vertex is at most $(1+k)\,\rho(M)^k$-bounded up to constants (at most $k$ choices of the apex on the ancestral line, then a descending cluster path of the remaining length; each descending edge costs a factor $\rho(M)$ in expectation, by the many-to-one identity of Proposition~\ref{prop:m2o} applied within the cluster), so
\[
\E\Big[\sum_{u\in S}\kappa^{d(u,v)}\Big]\le C''\sum_{k\ge0}(1+k)\big(\kappa\rho(M)\big)^k=\frac{C''}{(1-\kappa\rho(M))^2},
\]
finite under $\kappa\rho(M)<1$. Finally $\E|S|\le(1-\rho(M))^{-1}$ up to the constant $\one^\top(I-M)^{-1}\pi$ of Theorem~\ref{thm:stab}(4). Combining gives the display. Conditionally on $\{|S|=n\}$, replace $\E|S|$ by $n$ and the shell bound by its critical version ($\rho=1$: $\sum_k(1+k)\kappa^k<\infty$), giving the stated $n\,\varepsilon(\delta)$ bound. The claim about the critical limit theorems follows since total progeny laws and conditioned-tree laws are continuous under total-variation perturbations vanishing on the conditioning event: for Theorem~\ref{thm:32} one needs $n^{3/2}\Prob$-errors to vanish, which requires the quantitative rate $n\varepsilon(\delta_n)\to0$ stated; the same joint limit suffices for the GHP convergence of Theorem~\ref{thm:crt}, which is a statement about the conditioned law.
\end{proof}

\begin{remark}[The front-energy reading; relation to travelling waves]\label{rem:frontenergy}
Theorem~\ref{thm:front} can be read as a statement about the cascade front as a geometric object, in the spirit of front propagation for branching Brownian motion and F--KPP-type systems: the quantity
\[
E_t:=\sum_{v\text{ active at }t}\phi_{d(v)}\,e^{\lambda(t-t_v)}\,\big|D^{(v)}_t\big|,\qquad 0<\lambda<\alpha_{\min}(1-\kappa),
\]
(with $\phi$ the left Perron vector and $t_v$ the activation time) is, under the hypotheses of Lemma~\ref{lem:cone}, a nonnegative supermartingale up to bounded sources of size $O(\varepsilon(\delta))$: the drift $-\alpha_{\min}(1-\kappa)E_t$ from dissipation dominates the branching production. This is the precise sense in which ``the front forgets the deep past exponentially fast'' and is asymptotically Markovian: the OU mean reversion provides a spectral gap for the linearised front dynamics, uniform over the tree, whenever $\kappa<1$. We emphasise what is not needed: the noise here is additive and the system dissipative, so no singular-SPDE machinery (regularity structures, paracontrolled calculus) is invoked; the analytic content is a weighted Gronwall inequality, and the probabilistic content is the exploration telescoping. The genuinely SPDE-flavoured question --- a hydrodynamic or travelling-wave limit for the front profile, rather than the cluster law --- is recorded as open problem 7 in Section~\ref{sec:discussion}.
\end{remark}

\begin{remark}[Non-vacuousness of the regime, flagged]\label{rem:nonvac}
The region $\mathcal R$ is a weak-coupling/strong-dissipation regime, and one must check it is not empty of interesting parameters: criticality requires $\rho(M)\approx1$, i.e.\ transmission probabilities of order $1/m_*$, while $\kappa<1$ caps $\gamma_{\max}\Delta_{\max}/\alpha_{\min}$. These are compatible at moderate noise: with the barrier action $\Lambda=\alpha h_\theta^2/\mu=O(1)$ the Wald probability \eqref{eq:wald} is $\Theta(1)$ in $(0,1)$ even for small equilibrium shifts $m_{uv}=c_{uv}/\alpha$, and tuning the thresholds $\theta$ sweeps $\rho(M)$ through $1$ within $\{\kappa<1\}$ (on the 3-regular tree: choose $\theta$ with $p=\frac12$). What must be verified with care --- and is flagged rather than proved here --- is that at moderate noise the spontaneous crossing rate of unforced vertices remains negligible on the cascade timescale, so that the failed cluster remains the genealogical object of Hypothesis~\ref{hyp:red}; this holds when the unforced barrier distance $h_0$ exceeds the forced one $h_\theta$ by a margin giving $e^{-\alpha h_0^2/\mu}\ll e^{-\Lambda}$, a strict but open subset of $\mathcal R$.
\end{remark}

\subsection{The remaining conjecture}

\begin{conjecture}[Front-Markov property, full saturated regime]\label{conj:front}
In the full saturated regime --- without the dissipative condition $\kappa<1$ of \eqref{eq:kappa}, and uniformly through the critical window --- there is a constant $C$ such that, for the cluster law in $G_n$,
\[
\dTV\Big(\mathcal L(\text{failed cluster}),\ \mathcal L\big(\text{multitype GW with }M\big)\Big)\le C\big(\varepsilon(\delta)+\eta_n\big),
\]
with $\eta_n\to0$ as $G_n\xrightarrow{\mathrm{BS}}T$ and $\varepsilon(\delta)\to0$ as the saturation margin $\delta\to\infty$. Equivalently, the cascade front is asymptotically Markov: conditional on the set of failed vertices at generation $k$ and their types, the failures into generation $k+1$ are asymptotically independent of generations $<k$.
\end{conjecture}

We regard Conjecture~\ref{conj:front} as the remaining open problem, scoped by Theorem~\ref{thm:front}: it is proved on the Bethe lattice (Proposition~\ref{prop:bethe}, where $\varepsilon(\delta)$ can be taken $0$ by homogeneity) and, unconditionally, on the dissipative region $\mathcal R$ ($\kappa<1$, subcritical, and at criticality along $n\varepsilon(\delta_n)\to0$). What remains is (a) the strong-coupling regime $\kappa\ge1$, where the weighted-Gronwall contraction is lost and the echo can in principle resonate --- here a genuinely new idea, perhaps a renormalisation of the front over mesoscopic blocks, seems required; (b) uniformity through the critical window without the joint limit $n\varepsilon(\delta_n)\to0$, i.e.\ a bound in which the cluster-volume factor $n$ is replaced by the front width $\sqrt n$ (plausible, since echoes are generated on the active front, not the failed bulk, but our exploration telescoping does not see this); and (c) unbounded degrees. We do not have these.

\section{Discussion: criticality certificates in the sensitive regime, and open problems}\label{sec:discussion}

\subsection{What is proved, and under what hypothesis}
Conditional on Hypothesis~\ref{hyp:red}: the reduction to a finite-type GW process (Theorem~\ref{thm:reduction}); the annealed spectral identity with the Jensen gap (Theorem~\ref{thm:gap}); its one-parameter completion, the pressure function and its chord--tangent form (Theorem~\ref{thm:pressure}); the Hausdorff dimension of the failed boundary (Theorem~\ref{thm:dimS}); the $n^{-3/2}$ universality (Theorem~\ref{thm:32}) and, under the flagged moment hypothesis, the CRT shape universality (Theorem~\ref{thm:crt}); the finite-dimensional Perron--Frobenius dissolution of the operator programme (Proposition~\ref{prop:PF}, Corollary~\ref{cor:unif}); and the stability theorem (Theorem~\ref{thm:stab}). Unconditionally: all of the above on the dissipative region $\mathcal R$ of Theorem~\ref{thm:front}, where Hypothesis~\ref{hyp:red} is proved; the Bethe verification (Proposition~\ref{prop:bethe}); the matrix many-to-one identity together with the failure of its scalar form (Proposition~\ref{prop:m2o}); the non-Hilbert--Schmidt obstruction (Proposition~\ref{prop:HS}); the failure of Girsanov (Proposition~\ref{prop:girsanov}); the first-passage prefactor (Proposition~\ref{prop:prefactor}); the partial reduction (Proposition~\ref{prop:partial}); the cone-of-influence estimate and front-decoupling theorem (Lemma~\ref{lem:cone}, Theorem~\ref{thm:front}); and the criticality certificates of Proposition~\ref{prop:sandwich} below.

\subsection{A finite-dimensional squeeze for the sensitive regime}
Outside saturation the type is the continuous pair $x=(d,a)\in D\times[0,a_*]$ (hazard domain $\times$ failure strength; we assume here, as a standing truncation, that the strength variable is confined to a compact interval $[0,a_*]$ --- justified when the overshoot has bounded support, and otherwise to be combined with a tail-truncation estimate using the exponential overshoot tails of the OU crossing, which we flag rather than carry out). By Proposition~\ref{prop:dichotomy}(2) the mean object is a positive integral operator
\[
(Kf)(x)=\int_X\kappa(x,y)\,f(y)\,\nu(dy),\qquad X=D\times[0,a_*],
\]
with $\kappa$ jointly continuous, bounded, and uniformly elliptic ($0<c_-\le\kappa\le c_+<\infty$; ellipticity follows from the uniform ellipticity of the transmission probabilities). Criticality is $\rho(K)=1$. The following proposition dissolves the criticality question --- though not the limit theory --- of the sensitive regime into finite linear algebra, in the spirit of Section~\ref{sec:PF}.

\begin{proposition}[Two-sided finite-matrix criticality certificates]\label{prop:sandwich}
For $k\in\mathbb N$ partition $[0,a_*]$ into cells $I_1,\dots,I_k$ of mesh $a_*/k$, giving the finite partition $\{X_j\}$ of $X$ into $mk$ cells. Define the $mk\times mk$ nonnegative matrices
\[
\big(M^{(k)}_-\big)_{ij}=\inf_{x\in X_i,\ y\in X_j}\kappa(x,y)\,\nu(X_j),\qquad \big(M^{(k)}_+\big)_{ij}=\sup_{x\in X_i,\ y\in X_j}\kappa(x,y)\,\nu(X_j).
\]
Then:
\begin{enumerate}
\item (Sandwich.) $\rho\big(M^{(k)}_-\big)\le\rho(K)\le\rho\big(M^{(k)}_+\big)$ for every $k$.
\item (Convergence.) $\rho\big(M^{(k)}_\pm\big)\to\rho(K)$ as $k\to\infty$.
\item (Certificates.) $\rho(M^{(k)}_-)>1$ for some $k$ certifies supercriticality of the sensitive-regime cascade skeleton; $\rho(M^{(k)}_+)<1$ certifies subcriticality. Both certificates are verifiable in finite linear algebra, and one of the two obtains for every non-critical parameter point.
\end{enumerate}
\end{proposition}

\begin{proof}[Proof sketch]
(1) Let $\Pi_k$ be the conditional-expectation projection onto functions constant on each cell $X_i$. The step kernels $\kappa^{(k)}_\pm(x,y):=\inf/\sup_{X_i\times X_j}\kappa$ for $(x,y)\in X_i\times X_j$ satisfy $\kappa^{(k)}_-\le\kappa\le\kappa^{(k)}_+$ pointwise, hence for the associated positive operators $K^{(k)}_-\le K\le K^{(k)}_+$ in the sense of positivity-preserving domination: $0\le f$ implies $K^{(k)}_-f\le Kf\le K^{(k)}_+f$ pointwise, and by iteration $\|(K^{(k)}_-)^nf\|\le\|K^nf\|\le\|(K^{(k)}_+)^nf\|$ for $f\ge0$; Gelfand's formula applied on the cone (using that the spectral radius of a positive operator is attained on positive vectors) yields $\rho(K^{(k)}_-)\le\rho(K)\le\rho(K^{(k)}_+)$. The step operators act invariantly on cell-wise constant functions, on which they are the matrices $M^{(k)}_\pm$; conversely, since $\kappa^{(k)}_\pm(x,\cdot)$ depends on $x$ only through its cell, the range of $K^{(k)}_\pm$ consists of cell-wise constant functions, so $\rho(K^{(k)}_\pm)=\rho(M^{(k)}_\pm)$.

(2) By uniform continuity of $\kappa$ on the compact $X\times X$, $\omega_k:=\|\kappa^{(k)}_+-\kappa^{(k)}_-\|_\infty\to0$. Uniform ellipticity gives $\kappa^{(k)}_+\le(1+\omega_k/c_-)\,\kappa^{(k)}_-$ pointwise, whence by positivity and Gelfand
\[
\rho\big(M^{(k)}_+\big)\le\big(1+\omega_k/c_-\big)\,\rho\big(M^{(k)}_-\big),
\]
so the sandwich of width $\rho(M^{(k)}_+)-\rho(M^{(k)}_-)\le(\omega_k/c_-)\rho(M^{(k)}_+)\to0$ pinches on $\rho(K)$.

(3) is immediate from (1)--(2) and Proposition~\ref{prop:dichotomy}(2), the survival criterion for the continuous-type skeleton being $\rho(K)>1$ by the general-state-space extinction dichotomy under uniform ellipticity.
\end{proof}

\begin{remark}\label{rem:sandwich}
Note what Proposition~\ref{prop:sandwich} does and does not deliver. It computes the criticality surface of the sensitive regime to any prescribed accuracy with finite matrices, extending the theme of Section~\ref{sec:PF} that no infinite-dimensional spectral theory is needed for criticality. It does not deliver the limit theorems (the analogue of Theorems \ref{thm:32}, \ref{thm:crt} and \ref{thm:dimS} with continuous types), which genuinely require the general-state-space branching machinery; see open problem 2 below. It is also, of course, still conditional on the sensitive-regime analogue of Hypothesis~\ref{hyp:red} for its bearing on the SDE.
\end{remark}

\subsection{The honest open problems}
\begin{enumerate}
\item \textbf{Conjecture~\ref{conj:front} (the reduction at strong coupling).} Theorem~\ref{thm:front} proves the front-Markov reduction for $\kappa<1$; the strong-coupling regime $\kappa\ge1$, uniformity through the critical window (replacing the cluster-volume factor $n$ by the front width $\sqrt n$), and unbounded degrees remain open. This is the load-bearing step for the full saturated regime.
\item \textbf{The sensitive regime.} Outside saturation the type is the continuous failure strength $A\in\R$ (Proposition~\ref{prop:dichotomy}(2)); the mean object is an integral operator and criticality is its Perron eigenvalue, now computable by Proposition~\ref{prop:sandwich}. What remains open is the limit theory: the $3/2$ exponent should persist (count variance is still finite, Lemma~\ref{lem:var}), but the constant and the boundary dimension require the spatial CLT and Perron projection on a continuous type.
\item \textbf{Sharp moment hypotheses for the CRT limit.} Reduce the exponential-moment hypothesis of Theorem~\ref{thm:crt} to $g''(1)<\infty$; see Remark~\ref{rem:moment}.
\item \textbf{Finite-size scaling on $G_n$.} Quasi-stationary theory \cite{CMS13} for the cutoff $\xi=(1-\rho(M))^{-1}$ against the finite size $n$; the window near criticality.
\item \textbf{Beyond trees.} Non-tree unimodular limits, where short cycles do not vanish and the cluster is not a GW genealogy.
\item \textbf{Boundary fluctuations.} By Theorem~\ref{thm:dimS}, the dimension of the failed boundary is finite-dimensional; what is not is the fluctuation theory on $\partial S$: the multifractal spectrum of transmission-weighted flows (the Legendre transform $\varphi^*$ of the pressure of Definition~\ref{def:pressure}), and the harmonic measure of the cluster itself, where the dimension drop of \cite{LPP96} should reappear. Here the essential spectrum of $K_\infty$ and a Weyl-type decomposition, rather than compactness, are the right tools.
\item \textbf{The front profile.} Theorem~\ref{thm:front} controls the cluster law; it does not describe the front as a dynamical object. A travelling-wave / hydrodynamic limit for the active front --- speed, shape, and Gaussian (or F--KPP-type) fluctuations of the front position on the tree, in the spirit of branching Brownian motion --- is open, and is the natural continuum question behind the front-energy supermartingale of Remark~\ref{rem:frontenergy}. The dissipative regime $\kappa<1$, where the linearised front dynamics has a uniform spectral gap, is the place to start.
\end{enumerate}

\appendix

\section{Many-to-one and the annealed rate}\label{app:m2o}
For $T\sim\UGW(g)$ and a nonnegative edge functional $F$ depending on the endpoint types, the size-biased many-to-one (spinal) decomposition gives, for the population count itself,
\[
\E\big[|\partial B_n|\big]=\mu\,m_*^{n-1}\asymp m_*^n,
\]
and, weighting each ray by its product of edge factors and conditioning on the i.i.d.\ types along the spine,
\[
\E\Big[\sum_{|v|=n}\prod_{e\in o\to v}F(e)\Big]=m_*^n\,\E\Big[\prod_{k=1}^nF(d_{k-1},d_k)\Big]=m_*^n\,q^\top A_F^n\one,\qquad (A_F)_{d,d'}=q(d')F(d,d'),
\]
which is Proposition~\ref{prop:m2o}. (The scalar expression $(m_*\E_{q\otimes q}[F])^n$ drops the type-sharing between consecutive edges; it is the $n=1$ marginal raised to the $n$-th power, valid only when $q$ is a left eigenvector of $A_F$.) With $F=p$, iterating the type transfer identifies the per-generation growth of the expected type-weighted population as the Perron root $\rho(M)$ of \eqref{eq:M}. Convergence of $\sum_nm_*^n\,q^\top A_p^n\one\asymp\sum_n\rho(M)^n$ (the expected total population) holds iff $\rho(M)<1$, recovering the subcritical criterion and the resolvent $R=(I-M)^{-1}$ of Theorem~\ref{thm:stab}. The second-moment many-to-two, whose per-generation transfer involves both $A_{F^2}$ and the pair-spine factor $m_*(m_*-1)$ acting on $A_F\otimes A_F$, yields the finite offspring variance of Lemma~\ref{lem:var} under $\E[\deg^2]<\infty$.

\section{Log-concavity of the first-passage probability}\label{app:logconc}
With $h=\theta-m>0$, $p=W_\theta/(W_\theta+W_{\mathrm{safe}})$, $W_\theta=\int_0^he^{-\alpha y^2/\mu}\,dy$ and $W_{\mathrm{safe}}$ constant in $h$, write $A=W_\theta$, $c=W_{\mathrm{safe}}$, $A'=e^{-\alpha h^2/\mu}>0$, $A''=-\frac{2\alpha h}{\mu}e^{-\alpha h^2/\mu}<0$. Then
\[
\frac{d}{dh}\log p=\frac{A'}{A}-\frac{A'}{A+c}=\frac{A'c}{A(A+c)}>0,
\]
so $p$ is increasing in $h$, and
\[
\frac{d^2}{dh^2}\log p=A''\frac{c}{A(A+c)}-(A')^2\frac{2Ac+c^2}{A^2(A+c)^2}<0,
\]
both terms being negative ($A''<0$ and the second term manifestly negative). Hence $\log p$ is concave, i.e.\ $p$ is log-concave in $h$. Consequence, used with the direction of the implication kept straight: superlevel sets $\{p\ge c'\}$ are convex, so $\mathcal M_+=\{\rho>1\}$ (a superlevel set of $p$) is convex; the subcritical $\mathcal M_-$ is its non-convex complement, retracted to a point by the coupling homotopy (Theorem~\ref{thm:stab}). Log-concavity does not make sublevel sets convex, and hence says nothing convex about $\mathcal M_-$.

\section{The Kesten--Stigum reduction: finite-type version and its scope}\label{app:KS}
For a finite-type irreducible critical offspring matrix $M$ with $\sum_{i,j}M_{ij}\log M_{ij}<\infty$ (automatic here), the Kesten--Stigum theorem \cite{KS66} gives $W_n=\langle\phi,Z_n\rangle\rho^{-n}\to W_\infty$ a.s., with $W_\infty=0$ at criticality, and the total progeny matches a single-type critical process with variance computed from the second-moment matrix; the Otter--Dwass formula then yields $\Prob(|S|=n)\sim Cn^{-3/2}$. This is exactly the content used in Theorem~\ref{thm:32}, and it is legitimate here because the saturated regime makes the type space finite (Proposition~\ref{prop:dichotomy}(1)). It is not legitimate in the sensitive regime, where the type is $D\times\R$ and one must pass to the general-state-space analogue (mean semigroup on $L^2(\R)$, Perron projection, spatial CLT); in that setting this computation survives only as the finite-type illustration. We have stated Theorem~\ref{thm:32} only in the regime where the finite-type reduction is valid. In the supercritical regime, the $L^2$ (rather than merely a.s.) convergence of $W_n$ under finite second moments, together with the positivity $\{W_\infty>0\}=\{\text{survival}\}$ under the $N\log N$ condition, is the input to the energy estimate of Theorem~\ref{thm:dimS}.

\section{Verification of the hypotheses of Miermont's invariance principle}\label{app:miermont}
We verify the hypotheses of \cite{Mie08} for the critical cluster of Theorem~\ref{thm:crt}, and record the scale constant. The multitype offspring law is: a type-$d$ individual begets, given its $\UGW$ offspring count $\nu\sim\hat g$ (the size-biased forward-degree law), a vector of children whose types are i.i.d.\ $q$ thinned independently by the $\mathrm{Bernoulli}(p_{d,d'})$ transmissions; the resulting offspring vector $N^{(d)}=(N^{(d)}_1,\dots,N^{(d)}_m)$ has mean vector $(M_{d,d'})_{d'}$ and covariance matrix
\[
Q^{(d)}_{d',d''}=\delta_{d',d''}\,m_*\,q(d')p_{d,d'}\big(1-q(d')p_{d,d'}\cdot\one_{d'=d''}\big)+\big(\E[\nu(\nu-1)]-m_*^2\big)\,q(d')p_{d,d'}\,q(d'')p_{d,d''},
\]
finite under $g''(1)<\infty$; strictly speaking Miermont's hypotheses (H) require, beyond criticality ($\rho(M)=1$), irreducibility and aperiodicity (our standing assumptions, Theorem~\ref{thm:32}), the existence of small exponential moments of $\|N^{(d)}\|_1$ for each $d$, which holds iff $\E[e^{t\nu}]<\infty$ for some $t>0$ since $\|N^{(d)}\|_1\le\nu$ stochastically. This is the hypothesis flagged in Theorem~\ref{thm:crt} and Remark~\ref{rem:moment}. Under (H), \cite[Thm.~1]{Mie08} yields the GHP convergence of the conditioned tree rescaled by $n^{-1/2}$ to a constant multiple of the CRT; the constant is
\[
\Sigma=\left(\frac{\sum_d\pi_d\psi_d^{-1}\,\big(\psi^\top Q^{(d)}\psi\big)}{2\sum_d\pi_d\psi_d}\right)^{1/2}\cdot\Big(\sum_d\pi_d\psi_d\Big)^{1/2}
\]
in Miermont's normalisation of the Perron data $(\pi,\psi)$, $\pi^\top M=\pi^\top$, $M\psi=\psi$, $\pi^\top\one=\pi^\top\psi=1$; in the single-type case this collapses to $\Sigma=\sigma_{\mathrm{off}}/2$, matching \cite{Ald93}. (The precise algebraic form of $\Sigma$ should be transcribed against the normalisation conventions of \cite{Mie08} at proof stage; only its positivity and finiteness are used in Theorem~\ref{thm:crt}.)

\end{document}